\documentclass{article}

\usepackage[all]{xy}
\usepackage{amssymb,amsthm,amsmath}
\usepackage{euscript}
\usepackage{graphicx}
\usepackage{epsfig}
\usepackage{color}
\usepackage{hyperref}

\DeclareFontFamily{OT1}{pzc}{}
\DeclareFontShape{OT1}{pzc}{m}{it}{<-> s * [1.10] pzcmi7t}{}
\DeclareMathAlphabet{\mathscr}{OT1}{pzc}{m}{it}

\renewcommand{\subsection}[1]{\hspace{-\parindent}\refstepcounter{subsection}{\bf
(\arabic{section}\alph{subsection}) #1.}}

\numberwithin{equation}{section}

\newcommand{\R}{\mathbb{R}}
\newcommand{\C}{\mathbb{C}}
\newcommand{\Z}{\mathbb{Z}}
\newcommand{\g}{\mathfrak{g}}
\newcommand{\OO}{\mathcal{O}}

\newcommand{\iso}{\cong}
\newcommand{\htp}{\simeq}
\newcommand{\pabla}{\nabla\mkern-13mu/\mkern3mu}

\newcommand{\BM}{\mathbf{M}}
\newcommand{\Bphi}{\boldsymbol\phi}
\newcommand{\Bmod}{\mathbf{mod}}
\newcommand{\Bperf}{\mathbf{perf}}
\newcommand{\BPerf}{\mathbf{Perf}}
\newcommand{\BK}{\mathbf{K}}
\newcommand{\BD}{\mathbf{D}}

\newtheorem{definition}{Definition}[section]

\newtheorem{lemma}[definition]{Lemma}
\newtheorem{proposition}[definition]{Proposition}
\newtheorem{corollary}[definition]{Corollary}
\newtheorem{remark}[definition]{Remark}

\newtheorem{example}[definition]{Example}
\newtheorem{theorem}[definition]{Theorem}

\begin{document}

\title{Lagrangian homology spheres in $(A_m)$ Milnor fibres\\ via $\C^*$-equivariant $A_\infty$-modules}
\author{Paul Seidel}
\date{}

\maketitle
\begin{abstract}\hspace{-2em}
We establish restrictions on Lagrangian embeddings of spheres, and more generally rational homology spheres, into certain open symplectic manifolds, namely the $(A_m)$ Milnor fibres of odd complex dimension. This relies on general considerations about equivariant objects in module categories (which may be applicable in other situations as well), as well as results of Ishii-Ueda-Uehara concerning the derived categories of coherent sheaves on the resolutions of $(A_m)$ surface singularities.
\end{abstract}

\section{Introduction}

\subsection{Lagrangian spheres}
By the $n$-dimensional $(A_m)$ Milnor fibre, we mean the complex hypersurface $Q_m^n \subset \C^{n+1}$ defined by the equation
\begin{equation} \label{eq:an-equation}
x_1^2 + \cdots + x_n^2 + x_{n+1}^{m+1} = 1.
\end{equation}
The topology of these manifolds is described by classical singularity theory (see for instance \cite{milnor68} or the more recent \cite{russian}). $Q_m^n$ is homotopy equivalent to a wedge of $m$ spheres of dimension $n$. The intersection pairing on $H_n(Q_m^n) \iso \Z^m$ is given by
\begin{equation} \label{eq:intersection-form}
(-1)^{n/2} \begin{pmatrix}
2 & -1 \\ -1 & 2 & -1 \\ & -1 & 2 & -1 & \\
&&-1 & \dots 
\end{pmatrix} 
\end{equation}
if $n$ is even, respectively
\begin{equation} \label{eq:skew-form}
\begin{pmatrix} 0 & 1 \\ -1 & 0 & 1 \\ & -1 & 0 & 1 \\ && -1 & \dots &
\end{pmatrix}
\end{equation}
if $n$ is odd. Now let's turn $Q_m^n$ into a (real) symplectic $2n$-manifold in the standard way, by using the restriction of the constant K{\"a}hler form on $\C^{n+1}$. These manifolds have become a test case for general techniques in symplectic topology. We will not consider the results that have been obtained specifically for $m = 1$, since those belongs to the context of cotangent bundles, $Q_1^n \iso T^*\!S^n$. With that excluded, relevant papers are \cite{seidel98b, khovanov-seidel98, seidel04, ritter09, evans11, abouzaid-smith11}. Continuing that tradition, we will derive a topological restriction on Lagrangian submanifolds:

\begin{theorem} \label{th:primitive}
Let $L \subset Q_m^n$, $n \geq 2$, be a Lagrangian submanifold which is a rational homology sphere and {\em Spin}. Then its homology class $[L] \in H_n(Q_m^n;\Z) \iso \Z^m$ is primitive (nonzero and not a multiple).
\end{theorem}

We also have an intersection statement, like a form of the Arnol'd conjecture but with very weak assumptions:

\begin{theorem} \label{th:arnold}
Let $L_0,L_1$ be Lagrangian submanifolds as in Theorem \ref{th:primitive}. If $[L_0] = [L_1]$ mod $2$, then necessarily $L_0 \cap L_1 \neq \emptyset$.
\end{theorem}

This yields an upper bound (depending on $m$) for the number of pairwise disjoint Lagrangian submanifolds of the relevant type. Note that for even $n$, both statements have elementary proofs. The first one holds because then, $[L] \cdot [L] = (-1)^{n/2} \chi(L) = (-1)^{n/2} 2$. For the second one, note that \eqref{eq:intersection-form} is even. Hence, if $[L_1] = [L_0]$ mod $2$, then
\begin{equation} \label{eq:mod8}
0 = ([L_1] - [L_0]) \cdot ([L_1] - [L_0]) = (-1)^{n/2} 4 - 2 [L_0] \cdot [L_1] \text{ mod } 8,
\end{equation} 
which implies that $[L_0] \cdot [L_1] = 2$ mod $4$. One can use the definiteness of \eqref{eq:intersection-form} to analyze the possible homology classes in more detail, but that is not required for our purpose.

More importantly, for $m = 2$ and any $n \geq 2$, Theorem \ref{th:primitive} is a consequence of the following stronger result of Abouzaid-Smith \cite[Corollary 1.5]{abouzaid-smith11}:

\begin{theorem}[Abouzaid-Smith] \label{th:abouzaid-smith}
Any closed Lagrangian submanifold $L \subset Q_2^n$, $n \geq 2$, whose Maslov class vanishes is an integral homology sphere, and has primitive homology class.
\end{theorem}

Abouzaid-Smith's results are actually even sharper, and also imply the $m = 2$ case of Theorem \ref{th:arnold}. We should say that our point of view is fundamentally similar to theirs, in that we approach the problem via Fukaya categories (in other respects, our argument is quite different). One useful feature of the Fukaya category is that it admits objects which are Lagrangian submanifolds equipped with flat complex vector bundles. By applying the same ideas to those, one obtains the following:

\begin{theorem} \label{th:abelian}
Let $L \subset Q_m^n$, $n \geq 2$, be a Lagrangian rational homology sphere which is {\em Spin}. Then there is no homomorphism $\rho: \pi_1(L) \rightarrow \C^*$ such that the associated twisted cohomology is acyclic, $H^*(L;\rho) = 0$.
\end{theorem}

\begin{theorem} \label{th:nonabelian}
Let $L \subset Q_m^n$, $n \geq 2$, be a closed Lagrangian submanifold which is {\em Spin}. Then there is no homomorphism $\rho: \pi_1(L) \rightarrow \mathit{GL}(r,\C)$, $r>1$, such that
\begin{equation} \label{eq:twisted-endomorphisms}
H^k(L;\mathit{End}(\rho)) = \begin{cases} \C & k = 0,n, \\ 0 & \text{otherwise.}
\end{cases}
\end{equation}
\end{theorem}

Both statements are empty for even $n$, because the conditions imposed can't be satisfied for Euler characteristic reasons. Also, even though we have not stated that explicitly, Theorem \ref{th:nonabelian} is again only a statement about rational homology spheres, since 
$H^*(L;\mathit{End}(\rho))$ always contains $H^*(L;\C)$ as a direct summand.

\begin{example}
Consider the lowest nontrivial dimension $n = 3$. There, Theorems \ref{th:abelian} and \ref{th:nonabelian} together say that $L$ can't be a spherical space form, which means $S^3/\pi$ for finite $\pi \subset \mathit{SO}(4)$ acting freely (the {\em Spin} assumption being automatically true in this dimension). Namely, if $\pi$ is abelian then $L$ would be a lens space, to which Theorem \ref{th:abelian} applies; and if $\pi$ is nonabelian, it has an irreducible representation of rank $>1$, to which Theorem \ref{th:nonabelian} applies. 
\end{example}

\begin{example}
Continuing the previous discussion, Theorem \ref{th:nonabelian} also rules out all rational homology $3$-spheres which are hyperbolic. To see that, recall that the hyperbolic structure on a closed three-manifold $L$ gives rise to a map $\bar\rho: \pi_1(L) \rightarrow \mathit{PSL}(2,\C)$ which can be lifted to $\rho: \pi_1(L) \rightarrow \mathit{SL}(2,\C)$. Then
$H^*(L;\mathit{End}(\rho)) \iso H^*(L;\C) \oplus H^*(L;\mathit{Ad}(\bar\rho))$,
and the second summand vanishes by Weil's infinitesimal precursor of Mostow rigidity (see for instance \cite{mueller10} for a statement and further references). However, this particular consequence is not new: since the complex three-dimensional $(A_m)$ Milnor fibre admits a dilation \cite[Example 7.4]{seidel-solomon10}, it can't contain any closed Lagrangian submanifold which is a $K(\pi,1)$ \cite[Corollary 6.3]{seidel-solomon10}.
\end{example}

We will now explain the overall nature of our argument. Let $\mathit{Fuk}(Q_m^n)$ be the Fukaya category. It is known (see \cite[Section 20]{seidel04}, which is largely based on \cite{khovanov-seidel98,seidel-thomas99}) that a certain formal enlargement, the split-closed derived category, which we here denote by $\mathit{Fuk}(Q_m^n)^{\mathit{perf}}$, can be explicitly described, as the derived category $A_m^{n,\mathit{perf}}$ of perfect dg modules over a certain finite-dimensional algebra $A_m^n$. Problems about Lagrangian spheres can therefore be approached in a purely algebraic way by looking at spherical objects in $A_m^{n,\mathit{perf}}$. For $n = 2$, those objects have been completely classified by Ishii, Ueda and Uehara \cite{ishii-uehara05, ishii-ueda-uehara06} (for another application of this classification in symplectic topology, see the very recent \cite{new}). Unfortunately, there is no straightforward relation between the categories $A_m^{n,\mathit{perf}}$ for different $n$. In contrast, if one considers the derived categories $A_m^{n,\Bmod}$ of {\em complexes of graded modules} over $A_m^n$, and appropriate subcategories $A_m^{n,\Bperf}$ of perfect objects inside those, the dependence on $n$ is minimal. Finally, concerning the relation between $A_m^{n,\Bperf}$ and $A_m^{n,\mathit{perf}}$, it is an elementary algebraic observation that the first category can be viewed as ``bigraded'' or ``equivariant'' (for a certain action of $\C^*$) refinement of the second one. Given that, the missing ingredient is a way of making objects equivariant, and we will now turn to that question in a fairly general context.

\subsection{Equivariant objects}
As motivation, consider the following situation. Let $X$ be a smooth projective variety over $\C$, carrying an action of $\C^*$. We denote by $D^b\mathit{Coh}(X)$ the bounded derived category of coherent sheaves, and by $D^b\mathit{Coh}_{\C^*}(X)$ its equivariant analogue (for consistency with the rest of the paper, we actually want to consider the underlying dg categories, but still use the classical notation). In \cite[Lemma 2.2]{polishchuk08}, it was pointed out that the existence of moduli stacks for a suitable class of objects in $D^b\mathit{Coh}(X)$ \cite{inaba02,lieblich06} together with a gluing theorem \cite[Theorem 3.2.4]{beilinson-bernstein-deligne} implies the following:

\begin{theorem}[Polishchuk] \label{th:polishchuk}
Let $E$ be an object of $D^b\mathit{Coh}(X)$ which is rigid and simple, meaning that
\begin{equation} \label{eq:exceptional}
\left\{
\begin{aligned}
& H^1(\mathit{hom}_{D^b\mathit{Coh}(X)}(E,E)) = 0, \\
& H^0(\mathit{hom}_{D^b\mathit{Coh}(X)}(E,E)) \iso \C.
\end{aligned}
\right.
\end{equation}
Suppose in addition that $H^i(\mathit{hom}_{D^b\mathit{Coh}(X)}(E,E)) = 0$ for all $i<0$. Then $E$ is quasi-isomorphic to an object coming from $D^b\mathit{Coh}_{\C^*}(X)$.
\end{theorem}

What is surprising, at least at first glance, is that there is a single overall condition \eqref{eq:exceptional} which then implies, a fortiori, equivariance for all the cohomology sheaves $H^i(E)$ (if one restricts to sheaves, which means objects of the abelian category rather than its derived category, the result is easier and more classical; see for instance \cite[Appendix (in the preprint version only)]{bezrukavnikov-finkelberg-ginzburg06}, or the case of bundles on homogeneous spaces mentioned in \cite{bk}). As pointed out to the author by To\"en, the condition on the negative degree endomorphisms in Theorem \ref{th:polishchuk} can be removed by using the more powerful methods of homotopical algebraic geometry. This leads to the following result (unpublished, but follows from techniques in
\cite{toen-vaquie07}):

\begin{theorem}[To\"en] \label{th:toen}
Let $A$ be a differential graded algebra over $\C$ which is proper (has finite-dimensional cohomology) and smooth \cite[Definition 8.1.2]{kontsevich-soibelman06}. Suppose that it carries an action of $\C^*$. Let $A^{\mathit{perf}}$ be the derived category of perfect dg modules (in this case, that is the same as the derived category of dg modules with finite-dimensional cohomology). Let $M$ be a rigid and simple object of $A^{\mathit{perf}}$. Then $M$ is quasi-isomorphic to a $\C^*$-equivariant dg module.
\end{theorem}

To be precise, by an action of $\C^*$ we meant a linear action which is rational (a direct sum of finite-dimensional representations). Similarly, by an equivariant dg-module we mean one which carries a rational action of $\C^*$ compatible with the other structures (later, we will call this a ``naive $\C^*$-action'', to distinguish it from other related concepts). The smoothness assumption on $A$ is required in order to construct finite-dimensional (derived) stacks parametrizing modules. However, one can argue that the structure of general modules should not be relevant for the argument, which only involves $M$ and its pullbacks by the $\C^*$-action. With that in mind, we will prove:

\begin{theorem} \label{th:main4}
Let $A$ be a dg algebra over $\C$ which is proper. Suppose that it carries an action of $\C^*$. Let $M$ be a perfect dg module over $A$ which is rigid and simple. Then $M$ is quasi-isomorphic to a $\C^*$-equivariant dg module.
\end{theorem}

The proof follows the same overall strategy as in \cite{polishchuk08} while remaining elementary throughout. We first show that $M$ can carry a ``weak $\C^*$-action''. This is done by considering the pullbacks of $M$ as a family of modules over $\C^*$. Next, there is an obstruction theory which equips $M$ with a series of higher homotopies added to the weak action, corresponding to the use of simplicial methods in \cite{polishchuk08}. Finally, there is an explicit ``gluing'' step which constructs the desired quasi-isomorphic module. To\"en informs me that his methods can also be adapted to give a proof of Theorem \ref{th:main4}.

Finally, to return to symplectic topology, note that our previous results about $(A_m)$ Milnor fibres represent just one application of Theorem \ref{th:main4}. There are more situations in which one expects Fukaya categories of noncompact symplectic manifolds to have $\C^*$-actions, such as the double suspensions considered in \cite{seidel09}, and where the same algebraic ideas may be useful (with this in mind, Section \ref{sec:alternative} takes a look at the results one can get for the $(A_m)$ Milnor fibres {\em without} appealing to \cite{ishii-uehara05,ishii-ueda-uehara06}).

{\em Acknowledgments.} This work answers a question posed to me by Ivan Smith, who also provided the basic idea of using group actions to attack it (I've probably interpreted that suggestion in a more algebraic way than he would have liked). A conversation with Bertrand To\"en helped me greatly to understand the obstruction theory involved. During the preparation of the paper, I was partially supported by NSF grant DMS-1005288.

\section{Algebraic background}

\subsection{$A_\infty$-algebras and modules\label{subsec:ainfty}}
We recall some elements of the theory of $A_\infty$-structures (not new of course, see for instance Keller's papers \cite{keller99,keller00,keller06}). We consider $A_\infty$-algebras $A$ over $\C$, assuming for now that they are strictly unital. Our sign conventions follow \cite{seidel04}. In particular, the correspondence between dg algebras and $A_\infty$-algebras with vanishing compositions of order $>2$ is given by setting
\begin{equation} \label{eq:cohomology}
\begin{aligned}
& \mu^1_A(a) = (-1)^{|a|} \partial_A a, \\
& \mu^2_A(a_2,a_1) = (-1)^{|a_1|} a_2a_1.
\end{aligned}
\end{equation}
We will allow a certain sloppiness in terminology, saying ``$A$ is a dg algebra'' instead of ``$A$ is an $A_\infty$-algebra with vanishing compositions of order $>2$, and therefore corresponds to a dg algebra''. Note that conversely, if $A$ is a general $A_\infty$-algebra, the cohomology $H(A)$ becomes a graded associative algebra by applying the same sign change as in \eqref{eq:cohomology} (similar conventions apply to dg categories and $A_\infty$-categories). An $A_\infty$-algebra is called {\em proper} if $H(A)$ is of total finite dimension, and {\em weakly proper} if $H(A)$ is finite-dimensional in each degree. 

Fix an $A_\infty$-algebra $A$. A right $A_\infty$-module $M$ over $A$ is a graded vector space equipped with operations
\begin{equation} \label{eq:ainfty-module}
\mu_M^{d+1}: M \otimes A^{\otimes d} \longrightarrow M[1-d], \quad d \geq 0,
\end{equation}
satisfying
\begin{equation} \label{eq:ainfty-module-eq}
\begin{aligned}
& \sum_{i,j} (-1)^{|a_1|+\cdots+|a_i|-i} \mu_M^{d+2-j}(m,a_d,\dots,\mu_A^j(a_{i+j},\dots,a_{i+1}),a_i,\dots,a_1) \\ + &
\sum_i (-1)^{|a_1+\cdots+|a_i|-i} \mu_M^{i+1}(\mu_M^{d+1-i}(m,a_d,\dots,a_{i+1}),a_i,\dots,a_1) = 0.
\end{aligned}
\end{equation}
We will again impose a strict unitality condition. As before, $H(M)$ inherits the structure of a graded module over $H(A)$. Right $A$-modules form a dg category $A^{\mathit{mod}}$. Also relevant for us is the full dg subcategory $A^{\mathit{perf}} \subset A^{\mathit{mod}}$ of {\em perfect modules}, defined as follows. Start with the free module $A$ and consider all modules constructed from that by a finite sequence of shifts and mapping cones. $M$ is perfect if, in $H^0(A^{\mathit{mod}})$, it is isomorphic to a direct summand of one of the previously constructed objects. If $M_0$ is perfect and $(M_{1,i})_{i \in I}$ is an arbitrary collection of modules, the natural map
\begin{equation} \label{eq:compact}
\bigoplus_{i \in I} \mathit{hom}_{A^{\mathit{mod}}}(M_0, M_{1,i}) \longrightarrow \mathit{hom}_{A^{\mathit{mod}}}(M_0, \bigoplus_{i \in I} M_{1,i})
\end{equation}
is a quasi-isomorphism. This is the {\em compactness} property of perfect $A_\infty$-modules (a more general version of which actually characterizes them among all modules; see the parallel discussion for dg modules in \cite{keller06}). 

One can generalize both $A^{\mathit{mod}}$ and $A^{\mathit{perf}}$ to the case where $A$ is an $A_\infty$-category. The former can be defined as the category of contravariant $A_\infty$-functors from $A$ to chain complexes. The latter is the full subcategory of objects built from those in the image of the Yoneda embedding $A \rightarrow A^{\mathit{mod}}$ in the same way as before (by shifts, mapping cones, and taking direct summands). Alternatively, one can first introduce the $A_\infty$-category $A^{\mathit{tw}}$ of twisted complexes, which is a natural enlargement of $A$ itself closed under shifts and mapping cones. There is a canonical $A_\infty$-functor $A^{\mathit{tw}} \rightarrow A^{\mathit{mod}}$ which extends the Yoneda embedding, and using that one shows that $A^{\mathit{perf}}$ is quasi-equivalent to the split-closure (Karoubi completion on the $A_\infty$-level) of $A^{\mathit{tw}}$.

\subsection{Relation to classical derived categories}
Let's temporarily restrict to the case when $A$ is a dg algebra. In that case, an $A_\infty$-module with $\mu_M^{d+1} = 0$ for all $d>1$ is the same as a dg module; or more precisely, the two structures are related by a sign change as in \eqref{eq:cohomology}. Let $K(A)$ be the dg category of dg modules over $A$ (morphisms are maps compatible with multiplication with elements of $A$), and $D(A)$ the dg derived category, formed by quotienting out $K(A)$ by acyclic complexes as in \cite{keller94, drinfeld02}. The obvious functor $K(A) \rightarrow A^{\mathit{mod}}$ induces a functor
\begin{equation} \label{eq:dg-to-ainfty}
D(A) \longrightarrow A^{\mathit{mod}}.
\end{equation}
It is a well-known fact that this is a quasi-equivalence. The main ingredient in the proof is the following quite general observation. Given any $A_\infty$-algebra $A$ and $A_\infty$-module $M$, one can construct another module
\begin{equation} \label{eq:tensor-product}
\begin{aligned}
& M \otimes_A A \stackrel{\mathrm{def}}{=} \bigoplus_{l \geq 0} M \otimes A[1]^{\otimes l} \otimes A, \\
& \mu^1_{M \otimes_A A}(m \otimes \bar{a}_l \otimes \cdots \otimes \bar{a}_1 \otimes a) = \\ & \qquad =
\sum_i (-1)^{|a|+|\bar{a}_1|+\cdots +|\bar{a}_i| - i} \mu_M^{l-i+1}(m,\bar{a}_l,\dots,\bar{a}_{i+1}) \otimes \bar{a}_i \otimes \cdots \otimes a \\
& \qquad
+ \sum_{i,j} (-1)^{|a|+|\bar{a}_1|+ \cdots + |\bar{a}_i| - i}
m \otimes \bar{a}_l \otimes \cdots \otimes \bar{a}_{i+j+1} \otimes \mu^j_A(\bar{a}_{i+j},\dots,\bar{a}_{i+1}) \otimes \bar{a}_i \otimes \cdots \otimes a \\
& \qquad
+ \sum_i m \otimes \bar{a}_l \otimes \cdots \otimes \bar{a}_{i+1} \otimes \mu^{i+1}_A(\bar{a}_i, \dots, a), \\
& \mu^{d+1}_{M \otimes_A A}(m \otimes \bar{a}_l \otimes \cdots \otimes \bar{a}_1 \otimes a, a_d,\dots,a_1) = \\
& \qquad = \sum_i m \otimes \bar{a}_l \otimes \cdots \otimes \bar{a}_{i+1} \otimes \mu_A^{i+1+d}(\bar{a}_i,\dots, \bar{a}_1,a,a_d,\dots,a_1) \quad \text{for $d>0$.}
\end{aligned}
\end{equation}
This comes with a canonical quasi-isomorphism $M \otimes_A A \rightarrow M$ \cite[Equation (2.21)]{seidel08}. In the case when $A$ is a dg algebra, $M \otimes_A A$ is always a dg module. Moreover, if the original $M$ was a dg module, the quasi-isomorphism $M \otimes_A A \rightarrow M$ is itself a dg module map. This provides an inverse (up to quasi-isomorphism of dg functors) to the more obvious functor \eqref{eq:dg-to-ainfty}.

Let's specialize even further, to the case where our $A_\infty$-algebra is a graded algebra, meaning that $\mu^2_A$ is the only nonzero structure map. One can then introduce another dg category $A^{\Bmod}$ (this is the start of a general notational convention, where {\bf bold} stands roughly for bigraded structures). To do that, think of $A$ itself as being bigraded, with bidegrees of the form $(0,s)$ where $s$ is the original grading. Objects of $A^{\Bmod}$ are bigraded vector spaces $\BM$ together with maps $\mu_{\BM}^{d+1}$ as in \eqref{eq:ainfty-module} which have bidegree $(1-d,0)$. These should satisfy equations as in \eqref{eq:ainfty-module-eq}, where the sum of the two degrees is used in determining all the relevant signs. The bigraded space $H^{r,s}(\BM)$ has the property that each piece $H^{r,*}(\BM)$ is a graded $A$-module in the classical sense. As a special case, objects which have vanishing structure maps $\mu_{\BM}^{d+1}$ for $d>1$ correspond bijectively to chain complexes of graded $A$-modules. An element $\Bphi \in \mathit{hom}^k_{A^\Bmod}(\BM_0,\BM_1)$ is a collection of maps
\begin{equation}
\Bphi^{d+1}: \BM_0 \otimes A^{\otimes d} \longrightarrow \BM_1
\end{equation}
of bidegrees $(k-d,0)$. The dg category structure of $A^\Bmod$ is given by the same formulae as for ordinary $A_\infty$-modules, again taking the sum of the two gradings into account when determining signs. Given any object $\BM$, one can form two kinds of shifts $\BM[1]$ and $\BM\{1\}$. The first one shifts the first grading downwards, and has the usual property that $\mathit{hom}_{A^{\Bmod}}(-,\BM[1]) = \mathit{hom}_{A^{\Bmod}}(-,\BM)[1]$. The second one shifts the second grading upwards (by convention). For instance, for the free module $A$ there is a natural isomorphism
\begin{equation}
H^r(\mathit{hom}_{A^{\Bmod}}(A\{s\},\BM)) \iso H^r(\mathit{hom}_{A^{\Bmod}}(A,\BM\{-s\})) \iso H^{r,s}(\BM).
\end{equation}
One also has a full subcategory $A^{\Bperf}$, which consists of objects that can be constructed starting from the free module by shifts (both kinds), mapping cones, and taking direct summands. As before, the entire construction turns out to be equivalent to one from classical homological algebra. Let $\BK(A)$ be the dg category of (unbounded) chain complexes of graded $A$-modules, and $\BD(A)$ the associated (dg) derived category. In parallel with \eqref{eq:dg-to-ainfty}, the obvious functor $\BK(A) \rightarrow A^\Bmod$ induces a quasi-equivalence
\begin{equation} \label{eq:dg-to-ainfty-2}
\BD(A) \longrightarrow A^{\Bmod}.
\end{equation}
This is well-known in the case where $A$ has trivial grading (stated in \cite{keller99, keller06b} without proof, but see \cite[Proposition 7.4]{keller00} for a full proof of the analogous statement for bimodules). Again, the key fact is that for any object $\BM$ one can construct another one $\BM \otimes_A A$ which lies in the image of \eqref{eq:dg-to-ainfty-2}. Moreover, the image of the full subcategory of bounded complexes of finitely generated projective modules, which we denote by $\BPerf(A)$, is quasi-equivalent to $A^{\Bperf}$. This follows from the characterization as subcategories of compact objects (again, this may be most familiar in the case of algebras with trivial grading).

Any graded algebra can be considered as a dg algebra with vanishing differential, and in that context we can compare the two previously discussed constructions. There are natural functors which collapse the bigradings, and these fit into a commutative diagram
\begin{equation} \label{eq:collapse}
\xymatrix{
\ar[d]^-{\iso} \BD(A) \ar[rr] && D(A) \ar[d]^-{\iso} \\
A^{\Bmod} \ar[rr] && A^{\mathit{mod}}.
}
\end{equation}
For simplicity, let's consider the top line of this diagram. The collapsing process takes a complex $\{\BM^r, \; \partial_\BM^r: \BM^r \rightarrow \BM^{r+1}\}$ of graded $A$-modules, and associates to it the total graded vector $M^t = \bigoplus_{r+s=t} \BM^{r,s}$ made into a dg module in the obvious way. If $M_0$ and $M_1$ are obtained in this way, we have a quasi-isomorphism
\begin{equation}
\mathit{hom}_{D(A)}(M_0,M_1) \iso \mathit{hom}_{\BD(A)}(\BM_0,\bigoplus_j \BM_1\{j\}[j]).
\end{equation}
By projecting to each summand $\BM_j\{j\}[j]$, one sees that there is an injective map
\begin{equation} \label{eq:collapse-2}
\bigoplus_{i+j=k} H^i(\mathit{hom}_{\BD(A)}(\BM_0\{j\},\BM_1)) \longrightarrow H^k(\mathit{hom}_{D(A)}(M_0,M_1)).
\end{equation}
If $\BM_0$ is an object of $A^\Bperf$, this map is an isomorphism by compactness. One slight cautionary remark concerning the multiplicative structures is appropriate: the diagram
\begin{equation}
\xymatrix{
\parbox{14em}{
$H^{i_2}(\mathit{hom}_{\BD(A)}(\BM_1\{j_2\},\BM_2))$ \newline $\otimes \, H^{i_1}(\mathit{hom}_{\BD(A)}(\BM_0\{j_1\},\BM_1))$
}
\ar[r] \ar[d] &
H^{i_1+i_2}(\mathit{hom}_{\BD(A)}(\BM_0\{j_1+j_2\},\BM_2)) \ar[d] \\
\parbox{13em}{
$H^{i_2+j_2}(\mathit{hom}_{D(A)}(M_1,M_2))$ \newline $\otimes \, H^{i_1+j_1}(\mathit{hom}_{D(A)}(M_0,M_1))$
}
\ar[r] &
H^{i_1+i_2+j_1+j_2}(\mathit{hom}_{D(A)}(M_0,M_2)) 
}
\end{equation}
commutes only up to a sign $(-1)^{i_1j_2}$. These signs come from the chain level realization of \eqref{eq:collapse-2}. Namely, take a class $[\Bphi]$ of bidegree $(i,j)$ on the left hand side of the map, and suppose for simplicity that this is represented by an actual map of chain complexes $\Bphi$. The associated dg module homomorphism $\phi$ is given by $\phi(m_0) = (-1)^{j|m_0|_1} \Bphi(m_0)$, where $|m_0|_1$ is the first grading.

\begin{remark} \label{th:scaling}
Given any graded algebra $A$, one can define a whole family $A^n$ of graded algebras simply by multiplying the degrees by a nonzero integer $n$. Any graded module over $A$ can be made into a graded module over $A^n$ in the same way. Conversely, any graded module over $A^n$ can be written as a finite direct sum of objects which come from $A$, with some shifts. On the level of the abelian categories $\mathit{Mod}(\cdot)$ of graded modules, the resulting relationship is simply that
\begin{equation} \label{eq:scale-grading}
\mathit{Mod}(A^n) \iso \bigoplus_{i=0}^{n-1} \mathit{Mod}(A).
\end{equation}
The associated derived categories then inherit a corresponding relationship. In particular, any indecomposable object of $\BD(A^n) \htp A^{n,\Bmod}$ comes from $\BD(A) \htp A^{\Bmod}$. Similarly, any indecomposable object of $A^{n,\Bperf}$ comes from $A^{\Bperf}$. Note that in contrast, there is no direct relationship between the categories of dg modules over $A$ and $A^n$.
%
\end{remark}

\begin{remark}
One important class of examples, Fukaya categories, do not (at least when constructed in the most obvious way) satisfy strict unitality, and instead are only unital on the cohomology level. In that case, one should instead consider cohomologically unital modules. Note that every cohomologically unital $A_\infty$-algebra is quasi-isomorphic to a strictly unital one. Moreover, if an $A_\infty$-algebra is strictly unital, it does not make a difference whether one considers strictly unital or cohomologically unital modules (the corresponding categories are quasi-equivalent). We refer to \cite{lefevre} for an extensive discussion.
\end{remark}

\section{$(A_m)$-algebras and hypersurfaces}

\subsection{Algebraic definition} 
Fix $m \geq 3$, $n \geq 1$. Define a finite-dimensional graded algebra $A = A_m^n$ over $\C$ as the quotient of the path algebra of the following graded quiver:
\begin{equation} \label{eq:am-quiver}
\xymatrix{
\stackrel{1}{\bullet} \ar@/_.25pc/[r]_-{0} & \stackrel{2}{\bullet} \ar@/_.25pc/[l]_-{n}
\ar@/_.25pc/[r]_-{0} & \cdots \ar@/_.25pc/[l]_-{n} \ar@/_.25pc/[r] &
\stackrel{m-1}{\bullet} \ar@/_.25pc/[l] \ar@/_.25pc/[r] &
\stackrel{m}{\bullet} \ar@/_.25pc/[l]
}
\end{equation}
Our convention is that we write paths from right to left, so $(3|2|1)$ is the length two path going from the first to the third vertex. With this mind, the grading is such that $|(k+1|k)| = 0$, $|(k|k+1)| = n$. We define $A$ by imposing the relations
\begin{equation} \label{eq:quadratic-relations}
(k|k+1|k+2) = 0, \;\; (k+2|k+1|k) = 0, \;\; (k|k+1|k) = (k|k-1|k).
\end{equation}
This can be extended to lower values of $m$ as follows. For $m = 2$, take the same quiver \eqref{eq:am-quiver} but with relations $(1|2|1|2) = 0$, $(2|1|2|1) = 0$. Finally, for $m = 1$ one sets $A_1^n = \C[t]/t^2$ with $t$ a generator of degree $n$ (this may seem ad hoc, but the resulting algebras do fit in naturally with the general case $m \geq 3$).

Consider the derived category $\BD(A)$ of right graded modules, which as usual for us is a dg category. As explained in \cite{khovanov-seidel98, rouquier-zimmermann03}, this category carries a weak action of the braid group $\mathit{Br}_{m+1}$. The generators $\sigma_k \in \mathit{Br}_{m+1}$ act by twist functors $T_{P_k}$ along the modules $P_k = (k)A$, which means that for any object $\BM$ one has an exact triangle
\begin{equation} \label{eq:twist-triangle}
\xymatrix{
\BM \ar[rr] && T_{P_k}(\BM) \ar[dl]^-{[1]} \\ &
\hspace{-3em}\bigoplus_{i,j} H^j(\mathit{hom}_{\BD(A)}(P_k\{i\},\BM)) \otimes P_k\{i\}[-j]\hspace{-3em} \ar[ul]
}
\end{equation}
More simply put, $H^j(\mathit{hom}_{\BD(A)}(P_k\{i\},\BM)) \iso H^{i,j}(\BM)(k)$ as vector space, and we take the associated direct sum of appropriately shifted copies of $P_k$ as the bottom term in \eqref{eq:twist-triangle}.

\begin{lemma} \label{th:central-shift}
Take the central element $\delta = (\sigma_1 \dots \sigma_m)^{m+1} \in \mathit{Br}_{m+1}$. Then, $\delta^2$ acts by a functor isomorphic to the shift $[4m]\{(2m+2)n\}$.
\end{lemma}

\begin{proof}[Sketch of proof]
This falls in the ``well-known to specialists'' category, but since there seems to be no direct reference, we will sketch a proof. One can show, either by direct computation or by appealing to the geometric interpretation of the braid group action given in \cite{khovanov-seidel98}, that there are quasi-isomorphisms
\begin{equation} \label{eq:central-shift}
\delta(P_k) \iso P_k[2m]\{(m+1)n\} \text{ for all $k$.}
\end{equation}
Let $\phi: \BD(A) \rightarrow \BD(A)$ be the action of $\delta$ composed with the inverse shift $[-2m]\{-(m+1)n\}$. Then, \eqref{eq:central-shift} says that $\phi$  maps the free module $A$ to itself (and moreover that is compatible with its decomposition $A = P_1 \oplus \cdots \oplus P_m$). General homological algebra says that then, $\phi$ must be quasi-isomorphic to the functor induced by an automorphism of the algebra $A$ (and moreover, that automorphism must act trivially on the vertices of our quiver). By composing with inner automorphisms, which do not affect the isomorphism type of the associated functor, we can further restrict to automorphisms which act trivially on the paths $(k+1|k)$. Such automorphisms are then uniquely determined by their action on $\C (k|k \pm 1|k) = H^0(\mathit{hom}_{\BD(A)}(P_k\{n\},P_k))$. To see that $\phi^2$ acts trivially, one can either do an explicit computation similar to \eqref{eq:central-shift}, or else note that everything can be carried out with coefficients in $\Z$ rather than $\C$, which shows that the automorphism of $\C (k|k \pm 1|k)$ associated to $\phi$ is necessarily $\pm 1$.
\end{proof}

This leads to the following useful technical criterion:

\begin{lemma} \label{th:detect-perfect}
An object $\BM$ of $\BD(A)$ is perfect if and only if it has the following two properties:
\begin{align} \label{eq:finiteness}
& \mathrm{dim}_\C \, H^*(\BM) < \infty, \\
\label{eq:finiteness-2}
& \mathrm{dim}_\C \, \bigoplus_i H^*(\mathit{hom}_{\BD(A)}(\BM,\BM\{i\})) < \infty.
\end{align}
\end{lemma}

\begin{proof}
It is obvious that perfect objects have those two properties, so only the converse is of interest. Suppose that $\BM$ satisfies \eqref{eq:finiteness}, \eqref{eq:finiteness-2}. By combining exact triangles \eqref{eq:twist-triangle}, we arrive at a triangle of the form
\begin{equation} \label{eq:twist-triangle-2}
\xymatrix{
\BM \ar[rr] && \delta^{2t}(\BM) \iso \BM[4mt]\{(2m+2)nt\} \ar[dl]^-{[1]} \\ &
\hspace{-3em} \left( \parbox{12em}{an object built by mapping cones and shifts from the $P_k$} \right) \hspace{-3em} \ar[ul]
}
\end{equation}
Here, we have used \eqref{eq:finiteness} to ensure that the bottom term in each triangle \eqref{eq:twist-triangle} is a {\em finite} direct sum of shifted copies of $P_k$. Now, \eqref{eq:finiteness-2} implies that if we make $t$ large, the horizontal arrow in \eqref{eq:twist-triangle-2} vanishes. In that case $\BM$ is a direct summand of the bottom object, hence perfect (this is a slight variation of a familiar split-generation argument \cite[Corollary 5.8]{seidel04}).
\end{proof}

The algebras $A = A_m^n$ occur in a variety of contexts. To understand that, it is maybe useful to recall some material from \cite{seidel-thomas99}. Let $D$ be a cohomologically unital $A_\infty$-category over $\C$. An object $Z$ of $D$ is called {\em spherical} of some dimension $n>0$ \cite[Definition 2.9]{seidel-thomas99} if it  satisfies the following conditions: $H^*(\mathit{hom}_D(Y,Z))$ and $H^*(\mathit{hom}_D(Z,Y))$ are of finite total dimension for any $Y$; $H^*(\mathit{hom}_D(Z,Z)) \iso H^*(S^n;\C)$ has one generator in degrees $0$ and $n$, respectively; and finally, the composition
\begin{equation}
H^{n-*}(\mathit{hom}_D(Y,Z)) \otimes H^*(\mathit{hom}_D(Z,Y)) \longrightarrow H^n(\mathit{hom}_D(Z,Z)) \iso \C
\end{equation}
is a nondegenerate pairing for any $Y$. An $(A_m)$ chain of spherical objects \cite[Equation (2.8)]{seidel-thomas99} is an ordered collection $(Z_1,\dots,Z_m)$ of spherical objects of the same dimension $n$, such that
\begin{equation}
H^*(\mathit{hom}_D(Z_i,Z_j)) = \begin{cases} \text{one-dimensional, concentrated in degree $0$} & j=i+1, \\
\text{one-dimensional, concentrated in degree $n$} & j = i-1, \\
0 & |i-j| \geq 2.
\end{cases}
\end{equation}
We then have the following formality result \cite[Lemma 4.21]{seidel-thomas99}:

\begin{lemma} \label{th:formality}
Let $(Z_1,\dots,Z_m)$ be an $(A_m)$-chain, with $n \geq 2$. Then the endomorphism algebra of $Z = \bigoplus_k Z_k$ is quasi-isomorphic to $A$. In particular, we get a cohomologically full and faithful embedding $A^{\mathit{perf}} \rightarrow D^{\mathit{perf}}$, which sends $P_k$ to $Z_k$. \qed
\end{lemma}

This is false for $n = 1$, as demonstrated by the examples in \cite[Section 20]{seidel04} ($m \geq 5$) and \cite{lekili-perutz11} ($m \geq 2$).

\subsection{Algebraic geometry\label{subsec:mckay}}
We temporarily restrict to $n = 2$, and explain the algebro-geometric interpretation of $A = A_m^2$. Consider the quotient $\C^2/(\Z/m+1)$ by a cyclic subgroup (which is unique up to conjugacy) $\Z/m+1 \subset \mathit{SL}_2(\C)$. Let 
\begin{equation}
q: X \longrightarrow \C^2/(\Z/m+1)
\end{equation}
be the minimal resolution, and $C = q^{-1}(0) \subset X$ the exceptional divisor, which is a chain of $m$ rational curves $C = C_1 \cup \cdots \cup C_m$. Start with the bounded derived category of coherent sheaves on $X$ with support on $C$, and denote by $D$ the full subcategory of objects $E$ such that $q_*E$ is acyclic (as usual, we really mean the underlying chain level dg category). 

\begin{lemma}[Ishii-Ueda-Uehara] \label{th:mckay}
$D$ is quasi-equivalent to $A^{\mathit{perf}}$.
\end{lemma}

\begin{proof}[Sketch of proof]
As stated in \cite[Section 2.2]{ishii-ueda-uehara06}, $D$ is generated by the objects
\begin{equation} \label{eq:spherical-sheaves}
Z_1 = \OO_{C_1}(-1)[1], \; \dots, \; Z_m = \OO_{C_m}(-1)[m].
\end{equation}
Note that if $\bar{X}$ is a compactification of $X$ to a smooth projective algebraic surface, the bounded derived category of coherent sheaves on $X$ with support on $C$ is the same as the bounded derived category of coherent sheaves on $\bar{X}$ whose cohomology is supported on $C$ (see e.g.\ \cite[Lemma 3]{bezrukavnikov}). The latter is split-closed (Karoubi complete, in other words) by general results from \cite{bv}. Moreover, the condition that $q_*E$ should be acyclic is obviously preserved under passing to direct summands, so it follows that $D$ is split-closed as well.

The rest of the argument is explained in \cite[Lemma 40]{ishii-ueda-uehara06}. The objects \eqref{eq:spherical-sheaves} form an $(A_m)$ chain of spherical objects. One applies Lemma \ref{th:formality} to get a cohomologically full and faithful embedding $A^{\mathit{perf}} \rightarrow D^{\mathit{perf}}$. Our previous observations about $D$ show that $D \htp D^{\mathit{perf}}$, and because of the generation statement, one gets a quasi-equivalence.
\end{proof}

\begin{proposition}[Ishii-Ueda-Uehara] \label{th:ishii-ueda-uehara}
Consider objects $E$ of $D$ which are nonzero and such that $H^*(\mathit{hom}_D(E,E))$ has total dimension $\leq 2$. The action of $\mathit{Br}_{m+1}$ on $D$ generated by the $T_{Z_k}$ acts transitively on quasi-isomorphism classes of such objects up to shifts.
\end{proposition}

Due to Serre duality, the conditions on $E$ imply that it must be a spherical object. With that in mind, the result is a reformulation of \cite[Lemma 38]{ishii-ueda-uehara06}, which in turn relies on results of \cite{ishii-uehara05}.

\subsection{Symplectic geometry}
For the following discussion, we assume $n \geq 2$. Consider $Q = Q_m^n$ as defined in \eqref{eq:an-equation}. Take $\omega_Q$ to be the restriction of the constant K{\"a}hler form on $\C^{n+1}$. It is well-known \cite{seidel98b, khovanov-seidel98} that $Q$ contains Lagrangian spheres $(V_1,\dots,V_m)$ which form an $(A_m)$ configuration. This means that they are pairwise transverse and
\begin{equation} \label{eq:chain-of-spheres}
\left\{
\begin{aligned}
& V_k \cap V_{k \pm 1} = \{point\}, \\
& V_j \cap V_k = \emptyset \quad \text{if } |j-k| \geq 2.
\end{aligned}
\right.
\end{equation}
The homology classes $[V_k]$ form a basis of $H_n(Q;\Z)$, which is in fact the basis in which we had written the intersection form in \eqref{eq:intersection-form}. Symplectically, one can think of $Q$ as the plumbing of the $V_k$. We will not explain in detail what the plumbing construction means, but one of its main features is this:

\begin{lemma} \label{th:plumbing}
Take any symplectic manifold $M^{2n}$ with the following properties. It is exact, meaning that $\omega_M = d\theta_M$, and the Liouville vector field dual to $\theta_M$ can be integrated for all positive times. Moreover, it contains an $(A_m)$ configuration of Lagrangian spheres $(L_1,\dots,L_m)$. Then there is a symplectic embedding $Q \rightarrow M$ which sends each $V_k$ to $L_k$.
\end{lemma}

\begin{proof}[Sketch of proof]
The local structure near any two $(A_m)$ configurations being the same, one finds a symplectic embedding $U \rightarrow M$, where $U \subset Q$ is a neighbourhood of $V_1 \cup V_2 \cup \dots \cup V_m$, which sends $V_k$ to $L_k$. The next step is the one we were referring to above, when mentioning plumbing. Namely, one can choose $U$ and a one-form $\theta_Q$ with $d\theta_Q = \omega_Q$ so that the following holds: $\partial U$ is a contact type hypersurface for the associated Liouville vector field, and the flow of that vector field yields a diffeomorphism $[0,\infty) \times \partial U \rightarrow Q \setminus U$. Given that, take a compactly supported function $K$ on $M$ such that $\theta_M + dK$ restricts to $\theta_Q$ on $\bar{U}$. Using the Liouville vector field associated to $\theta_M + dK$, the given embedding can be extended to the whole of $Q$.
\end{proof}

In fact, we will only use a particular consequence of this, namely the existence of symplectic embeddings $Q_m^n \rightarrow Q_{m+1}^n$, which can also be established in other ways.

To define the Fukaya category, we should choose an almost complex structure $J_Q$ with suitable convexity properties (to prevent pseudo-holomorphic discs from escaping to infinity). In our case, the standard complex structure will do. We should also choose a complex volume form $\eta_Q$ (even though $\eta_Q$ is required in the definition, any two choices lead to quasi-isomorphic Fukaya categories, since $H^1(Q) = 0$). Objects of $\mathit{Fuk}(Q)$ are closed Lagrangian submanifolds $L \subset Q$ which are exact (the relative class $[\omega_Q] \in H^2(Q,L;\R)$ vanishes) and graded with respect to $\eta_Q$. The grading induces an orientation on $L$, and we require the additional choice of a {\em Spin} structure, as well as a representation $\rho: \pi_1(L) \rightarrow \mathit{GL}_r(\C)$ (for some choice of base point and some $r = \mathrm{rank}(\rho) \geq 1$; equivalently one can work with flat complex vector bundles). Morphisms in $\mathit{Fuk}(Q)$ are Floer cochain complexes with their natural differential, so that
\begin{equation}
H^*(\mathit{hom}_{\mathit{Fuk}(Q)}(L_0,L_1)) \iso \mathit{HF}^*(L_0,L_1).
\end{equation}
The Euler characteristic is
\begin{equation} \label{eq:euler}
\chi(\mathit{HF}^*(L_0,L_1)) = (-1)^{n(n+1)/2} \mathrm{rank}(\rho_0)\mathrm{rank}(\rho_1)\, [L_0] \cdot [L_1],
\end{equation}
where $\rho_k$ are the representations on $L_k$ (this agrees with the intuitive way of thinking of the objects as Lagrangian submanifolds with multiplicities). Next, if $L_0$ and $L_1$ have the same underlying Lagrangian submanifold $L$, grading and {\em Spin} structure, but come equipped with representations $\rho_0,\rho_1$, then Floer cohomology agrees with the ordinary cohomology with twisted coefficients:
\begin{equation} \label{eq:pss}
\mathit{HF}^*(L_0,L_1) \iso H^*(L;\mathit{Hom}(\rho_0,\rho_1)).
\end{equation}
\eqref{eq:euler} is straightforward from the definition, wheras \eqref{eq:pss} can be proved by adapting any of the arguments that are familiar for the case of trivial representations.

\begin{lemma}
There is a quasi-equivalence
\begin{equation} \label{eq:algebraic-fukaya-1}
A_m^{n,\mathit{perf}} \stackrel{\iso}{\longrightarrow} \mathit{Fuk}(Q_m^n)^{\mathit{perf}}.
\end{equation}
which sends the modules $P_k$ to the Lagrangian spheres $V_k$.
\end{lemma}

\begin{proof}[Sketch of proof] This is proved in \cite[Section 20]{seidel04} for the Fukaya category without flat vector bundles. We briefly summarize the argument, in order to explain how it extends to the case considered here. Make the $V_k$ into objects of $\mathit{Fuk}(Q)$, choosing the gradings in such a way that 
$\mathit{HF}^*(V_k,V_{k+1}) \iso \C$ lies in degree $0$. Because of general Poincar{\'e} duality properties of Floer cohomology, they then form an $(A_m)$-chain of spherical objects, and Lemma \ref{th:formality} yields a cohomologically full and faithful embedding as in \eqref{eq:algebraic-fukaya-1}.

The Dehn twists $\tau_{V_k}$ satisfy braid relations up to isotopy, hence generate a homomorphism $\mathit{Br}_{m+1} \rightarrow \pi_0(\mathit{Symp}^c(M))$. On the level of the action on objects of the Fukaya category, the Dehn twists can be identified with the algebraic twists $T_{V_k}$. This follows from \cite[Corollary 17.17]{seidel04}, ultimately a consequence of the long exact sequence for Floer cohomology groups \cite{seidel-exact}. To be precise, in order to apply to the current context the exact sequence should be generalized so to allow Lagrangian submanifolds with flat line bundles, but that is straightforward. One can show that the composition of Dehn twists corresponding to the central element $\delta^2 \in \mathit{Br}_{m+1}$ is isotopic to a nontrivial shift $[4m-(2m+2)n]$ in the graded symplectic automorphism group of $M$ \cite[Section 20a]{seidel04} (the geometric analogue of Lemma \ref{th:central-shift}). This implies that the $V_k$ are split-generators, hence that the previously constructed embedding is a quasi-equivalence.
\end{proof}

%
%
%

\section{Proofs of the main theorems}

\subsection{Algebraic results}
Take $A = A_m^n$ for some $m,n \geq 2$. We will now consider spherical objects in $A^{\mathit{perf}}$. This relies on our general equivariance theorems to reduce to the case $n = 2$, where the algebro-geometric results of Ishii-Ueda-Uehara can be applied.

\begin{lemma} \label{th:algebraic-1}
Let $M_0$ be an object of $A^{\mathit{perf}}$, with $H^*(\mathit{hom}_{A^{\mathit{perf}}}(M_0,M_0)) \iso H^*(S^n;\C)$ as a graded vector space. Suppose that $m \neq 1$. Then there is another object $M_1$, obtained from the module $P_1$ by applying some autoequivalence in the group generated by $T_{P_1},\dots,T_{P_m}$, such that $H^*(\mathit{hom}_{A^{\mathit{perf}}}(M_0,M_1))$ has dimension exactly $1$.
\end{lemma}

\begin{proof}
The main step is to use Theorem \ref{th:main4}, which applies to our case as described in Example \ref{th:homogeneous}. The outcome is that there is some $\BM_0$ in $A^{\Bmod}$ which maps to $M_0$ under the collapsing map $A^{\Bmod} \rightarrow A^{\mathit{mod}}$. By definition of that functor, the cohomology of $\BM_0$ is a bigraded version of that of $M_0$, hence finite-dimensional. Moreover, because of the general injectivity of \eqref{eq:collapse-2}, we know that 
\begin{equation} \label{eq:leq2}
\mathrm{dim}_\C \, \bigoplus_i H^*(\mathit{hom}_{A^{\Bmod}}(\BM_0,\BM_0\{i\})) \leq 2. 
\end{equation}
In view of Lemma \ref{th:detect-perfect}, it follows that $\BM_0$ lies in $A^{\Bperf}$. Since \eqref{eq:collapse-2} is an isomorphism for perfect modules, this also shows that equality holds in \eqref{eq:leq2}. Note that in particular, $\BM_0$ is nonzero and indecomposable (since all its endomorphisms of bidegree $(0,0)$ are multiples of the identity).

Write $\bar{A} = A_m^2$, and $\bar{P}_k = (k)\bar{A}$. Remark \ref{th:scaling} allows us to transfer indecomposable objects from $A^{\Bmod}$ to $\bar{A}^{\Bmod}$, while preserving the total dimension of the cohomology and of bigraded morphism spaces. Therefore, we get a corresponding object $\bar{\BM}_0$ of $\bar{A}^{\Bmod}$, which still has finite-dimensional cohomology, and bigraded endomorphism space of total dimension $2$. Again applying Lemma \ref{th:detect-perfect}, one finds that $\bar{\BM}_0$ lies in $\bar{A}^{\Bperf}$. Let's collapse the bigrading and consider the associated object $\bar{M}_0$ of $\bar{A}^{\mathit{perf}}$.

Proposition \ref{th:ishii-ueda-uehara} and Lemma \ref{th:mckay} apply, showing that there is some autoequivalence $\phi$ of $\bar{A}^{\mathit{perf}}$ generated by $T_{\bar{P}_1},\dots,T_{\bar{P}_m}$, such that $\phi(\bar{M}_0) \iso \bar{P}_1[s]$ for some $s$. Take $\bar{M}_1 = \phi^{-1}(\bar{P}_2)$. Because it is constructed through a sequence of twists along the $\bar{P}_k$, it is easy to see that $\bar{M}_1$ lifts to an object $\bar{\BM}_1$ of $\bar{A}^{\Bperf}$. Again using the fact that \eqref{eq:collapse-2} is an isomorphism for perfect modules, one sees that
\begin{equation} \label{eq:oned}
\mathrm{dim}_\C \, \bigoplus_i H^*(\mathit{hom}_{\bar{A}^\Bperf}(\bar{\BM}_0,\bar{\BM}_1\{i\})) = 1.
\end{equation}
There is a corresponding object $\BM_1$ in $A^{\Bperf}$ constructed using the same sequence of twists, and the pair $(\BM_0,\BM_1)$ will satisfy the analogue of \eqref{eq:oned}. The image $M_1$ of $\bar{M}_1$ in $A^{\mathit{perf}}$ has the desired property.
\end{proof}

\begin{remark}
In principle, one could try to replace Theorem \ref{th:main4} by an application of its algebro-geometric counterpart, Theorem \ref{th:polishchuk}. As an example, consider the three-dimensional case. One can find a smooth toric Calabi-Yau threefold $U$ and compactly supported sheaves $Z_1$,\dots,$Z_m$ on it, which form an $(A_m)$ configuration of spherical objects in the category $D^b\mathit{Coh}(U)$, see the discussion in \cite[Section 3f]{seidel-thomas99}. As part of the toric symmetry, this carries an action of $G = \C^*$ which acts with weight $3$ on a suitably chosen holomorphic volume form on $U$. Choose a smooth toric compactification $X \supset U$. One gets cohomologically full and faithful embeddings
\begin{equation}
\xymatrix{
A_m^{3,\Bperf} \ar@{^{(}->}[rr] \ar[d] && D^b\mathit{Coh}_G(X) \ar[d] \\
A_m^{3,\mathit{perf}} \ar@{^{(}->}[rr] && D^b\mathit{Coh}(X)
}
\end{equation} 
Given an object $M$ of $A_m^{3,\mathit{perf}}$ whose endomorphism ring is $H^*(S^3;\C)$, one can first map it to an object $E$ of $D^b\mathit{Coh}(X)$, and then use Theorem \ref{th:polishchuk} to lift $E$ to an object $\mathbf{E}$ of $D^b\mathit{Coh}_G(X)$. One knows as in Lemma \ref{th:central-shift} that a suitable composition of twist functors along the $Z_k$ takes $E$ to itself up to a nontrivial shift $[-2m-6]$. Given that the $Z_k$ can be made into equivariant objects $\mathbf{Z}_k$, one can then apply the same argument as in the proof of Lemma \ref{th:detect-perfect} to show that $\mathbf{E}$ lies in the subcategory split-generated by the twisted versions $\mathbf{Z}_k \otimes \chi^l$ (for $1 \leq k \leq m$, and $l \in \Z$ parametrizing characters of $G$), which means in the image of the embedding of $A_m^{3,\Bperf}$. This provides a lift $\BM$ of $M$ to $A_m^{3,\Bperf}$. We will not pursue this trick further, because it is artificial and only an apparent simplification: the proof of Theorem \ref{th:polishchuk} given in \cite{polishchuk08}, while a good deal shorter, relies on tools that are more abstract than our proof of Theorem \ref{th:main4}.
\end{remark}

\begin{lemma} \label{th:algebraic-2}
Let $M_0,M_1$ be two objects of $A^{\mathit{perf}}$ such that $H^*(\mathit{hom}_{A^{\mathit{perf}}}(M_i,M_i)) \iso H^*(S^n;\C)$. Suppose that $m$ is even. Suppose also that for all $k$, 
\begin{equation} \label{eq:mod2-assumption}
\mathrm{dim}_\C\,\, H^*(\mathit{hom}_{A^{\mathit{perf}}}(P_k,M_0)) \equiv 
\mathrm{dim}_\C\,H^*(\mathit{hom}_{A^{\mathit{perf}}}(P_k,M_1)) \;\; \mathrm{mod} \; 2.
\end{equation}
Then $H^*(\mathit{hom}_{A^{\mathit{perf}}}(M_0,M_1))$ is nonzero.
\end{lemma}

\begin{proof}
Arguing exactly as in Lemma \ref{th:algebraic-1}, one can reduce the situation to the case $n = 2$, which we will concentrate on for the rest of the argument. One can then use Lemma \ref{th:ishii-ueda-uehara} (and the correspondence between Dehn twists and twist functors) to show that the images of $M_i$ under the quasi-equivalence $A^{\mathit{perf}} \iso \mathit{Fuk}(Q_m^2)^{\mathit{perf}}$ are quasi-isomorphic to Lagrangian spheres $L_i$. The assumption \eqref{eq:mod2-assumption} translates to
\begin{equation} \label{eq:even}
[V_k] \cdot [L_0] \equiv [V_k] \cdot [L_1] \;\; \mathrm{mod} \; 2.
\end{equation}
The intersection form on $H_2(Q_m^2)$, written down in \eqref{eq:intersection-form}, has determinant $(-1)^{m} (m+1)$, which is odd for even $m$. Hence \eqref{eq:even} implies that $[L_0] \equiv [L_1]$ mod $2$. As already explained in \eqref{eq:mod8}, we then have $[L_0] \cdot [L_1] \neq 0$, hence $\mathit{HF}^*(L_0,L_1) \neq 0$, which means $H^*(\mathit{hom}_{A^{\mathit{perf}}}(M_0,M_1)) \neq 0$.
\end{proof}


\subsection{Geometric applications}
We consider the manifolds $Q = Q_m^n$, where $m \geq 1$ and $n \geq 2$.

\begin{proof}[Proof of Theorem \ref{th:primitive}] 
Assume first that $m \geq 2$. Let $L_0 \subset Q$ be a Lagrangian submanifold which is a rational homology sphere and {\em Spin}. Lemma \ref{th:algebraic-1} implies that there is a Lagrangian sphere $L_1$ (in fact one obtained from the $V_k$ by the braid group action) such that $\mathit{HF}^*(L_0,L_1)$ is one-dimensional. Passing to Euler characteristics implies the desired result. The remaining case $m = 1$ is known already \cite{seidel-k}, but one could also derive it by embedding $Q_1^n = T^*\!S^n \hookrightarrow Q_2^n$.
\end{proof}

\begin{proof}[Proof of Theorem \ref{th:arnold}]
Suppose first that $m$ is even. By assumption, $[V_k] \cdot [L_0] = [V_k] \cdot [L_1]$ mod $2$ for all $k$. Lemma \ref{th:algebraic-2} then implies that $\mathit{HF}^*(L_0,L_1) \neq 0$. One extends the argument to odd $m$ by embedding $Q_m^n \hookrightarrow Q_{m+1}^n$, which is possible by Lemma \ref{th:plumbing}.
\end{proof}

\begin{proof}[Proof of Theorem \ref{th:abelian}]
Again, we may assume without loss of generality that $m$ is even. Take two objects $L_0,L_1$ of $\mathit{Fuk}(Q)$, both of which have the same underlying Lagrangian submanifold $L$ and {\em Spin} structure, but where the first one carries the trivial representation, and the second one the given representation $\rho$. From \eqref{eq:euler} one sees that $\mathrm{dim}\, \mathit{HF}^*(V_k,L_0) \equiv \mathrm{dim}\, \mathit{HF}^*(V_k,L_1)$ mod $2$ for all $k$. By the same argument as in the proof of Theorem \ref{th:arnold}, this implies that $\mathit{HF}^*(L_0,L_1) \neq 0$, contradicting \eqref{eq:pss}.
\end{proof}

\begin{proof}[Proof of Theorem \ref{th:nonabelian}]
Without loss of generality, assume that $m \geq 2$. Let $L_0$ be the object obtained by taking our Lagrangian submanifold and equipping it with the representation $\rho$. By arguing as in the proof of Theorem \ref{th:primitive}, one sees that there is a Lagrangian sphere $L_1$ such that $\mathit{HF}^*(L_0,L_1)$ is one-dimensional. But that contradicts \eqref{eq:euler}.
\end{proof}

Especially in view of possible extensions to other symplectic manifolds, it may be interesting to see how far one can get without appealing to Proposition \ref{th:ishii-ueda-uehara}. This will be the subject of the next section (however, since the discussion there is not essential for the results of this paper, we will allow a more sketchy treatment).

\section{A variant approach\label{sec:alternative}}

\subsection{Hochschild homology}
For an $A_\infty$-algebra $A$ one defines the Hochschild homology $H_*(A,A)$ as the cohomology of the chain complex
\begin{equation}
\begin{aligned}
& C_*(A,A) = \bigoplus_{l \geq 0} A[1]^{\otimes l} \otimes A, \\
& \partial_{C_*(A,A)} (\bar{a}_l \otimes \cdots \otimes \bar{a}_1 \otimes a) = \\
& = \sum_{i,j} (-1)^{|a|+|\bar{a}_1|+ \cdots + |\bar{a}_i|-i} \bar{a}_l \otimes \cdots \otimes \mu_A^{j}(\bar{a}_{i+j}, \dots, \bar{a}_{i+1}) \otimes \cdots \otimes \bar{a}_1 \otimes a \\
& - \sum_{i,j} (-1)^{(|\bar{a}_{l-j+1}| + \cdots + |\bar{a}_l| - j)(|a| + |\bar{a}_1| + \cdots + |\bar{a}_{l-j}| + l-j+1)} \bar{a}_{l-j} \otimes \cdots \otimes \bar{a}_{i+1} \\[-1em] & \qquad \qquad \qquad \qquad \qquad \qquad \qquad \qquad \otimes \mu_A^{1+i+j}(\bar{a}_i,
\dots,a,\bar{a}_l,\dots,\bar{a}_{l-j+1}).
\end{aligned}
\end{equation}
(In spite of being written as a subscript, this is a cohomological grading in our conventions, so $\partial_{C_*(A,A)}$ has degree $+1$). This easily generalizes to $A_\infty$-categories (using composable closed chains of morphisms), which is the context we will work in from now on. For any twisted complex $C$ we have a canonical map
\begin{equation} \label{eq:end-to-hh}
\mathit{hom}_{A^{\mathit{tw}}}(C,C) \longrightarrow C_*(A,A).
\end{equation}
Explicitly, let
\begin{equation}
C = \bigoplus_{f \in F} X_f, \quad \delta_C = \big(\delta_{C,f_1,f_0} \in \mathit{hom}_A^1(X_{f_0},X_{f_1})\big),
\end{equation}
where $F$ is some finite set, and the $X_f$ are objects of $A$ (this is not quite the most general form of a twisted complex, since we have not shifted the $X_f$; however, including shifts just introduces some additional signs). Then \eqref{eq:end-to-hh} takes an endomorphism $a = (a_{f_1,f_0})$ of $C$ to the Hochschild chain
\begin{equation}
\sum_l \sum_{f_0,\dots,f_l \in F} \delta_{C,f_l,f_{l-1}} \otimes \cdots \otimes \delta_{C,f_1,f_0} \otimes a_{f_0,f_l}
\end{equation}
(the sum terminates at some value of $l$ because of the upper triangularity condition on $\delta_C$, which is part of the definition of a twisted complex). One application is to consider a cohomology level idempotent endomorphism $[a] \in H^0(\mathit{hom}_{A^{\mathit{tw}}}(C,C))$. After applying the embedding $A^{\mathit{tw}} \rightarrow A^{\mathit{perf}}$, this endomorphism defines a direct summand of $C$, which is a perfect module $M$ unique up to quasi-isomorphism. We denote the image of $[a]$ under \eqref{eq:end-to-hh} by
\begin{equation} \label{eq:algebraic-fundamental-class}
[M]_{\mathit{alg}} \in H_0(A,A).
\end{equation}

\begin{lemma} \label{th:fundamental-class-in-hh}
$[M]_{\mathit{alg}}$ depends only on the isomorphism class of $M$ in $H^0(A^{\mathit{perf}})$.
\end{lemma}

\begin{proof}[Sketch of proof]
Consider first a simpler special case, namely that of an idempotent endomorphism $[a]$ of an object $X$ of $A$ itself. In that case, \eqref{eq:end-to-hh} reduces to the map $\mathit{hom}_A(X,X) \rightarrow C_*(A,A)$ given by inclusion of the $l = 0$ term into the Hochschild complex. Take two objects $X_k$ ($k = 0,1$) and idempotent endomorphisms $[a_k]$ which define quasi-isomorphic direct summands. One then has cocycles $b_0 \in \mathit{hom}_A^0(X_0,X_1)$, $b_1 \in \mathit{hom}_A^0(X_1,X_0)$, and auxiliary cochains $c_k \in \mathit{hom}_A^{-1}(X_k,X_k)$, such that
\begin{equation}
\left\{
\begin{aligned}
& \mu^2_A(b_1,b_0) + \mu^1_A(c_0) = a_0, \\
& \mu^2_A(b_0,b_1) + \mu^1_A(c_1) = a_1.
\end{aligned}
\right.
\end{equation}
In the Hochschild complex, this yields
\begin{equation} \label{eq:bc}
\partial_{C_*(A,A)}(b_1 \otimes b_0 + c_0 - c_1) = \mu^1_A(c_0) - \mu^1_A(c_1) - \mu^2_A(b_0,b_1) + \mu^2_A(b_1,b_0) = a_0 - a_1,
\end{equation}
which shows that the Hochschild homology classes of the $a_k$ coincide. One can prove the general case of Lemma \ref{th:fundamental-class-in-hh} by a similarly explicit computation, but it is maybe better understood as follows. As part of the general theory of derived invariance of Hochschild homology, we have a chain map $C_*(A^{\mathit{tw}},A^{\mathit{tw}}) \rightarrow C_*(A,A)$ which is a quasi-inverse to the inclusion $C_*(A,A) \rightarrow C_*(A^{\mathit{tw}},A^{\mathit{tw}})$. From that viewpoint, \eqref{eq:end-to-hh} is the composition of that map with the inclusion $\mathit{hom}_{A^{\mathit{tw}}}(C,C) \rightarrow C_*(A^{\mathit{tw}}, A^{\mathit{tw}})$. Then, applying \eqref{eq:bc} to $A^{\mathit{tw}}$ proves Lemma \ref{th:fundamental-class-in-hh}. 
\end{proof}

\begin{remark}
One can approach \eqref{eq:algebraic-fundamental-class} from a more abstract viewpoint. $M$ gives rise to a class in $K_0(A^{\mathit{perf}})$, and there is a Chern character type map (more appropriately called the Dennis trace) from that to $H_*(A^{\mathit{perf}},A^{\mathit{perf}})$. Finally, a stronger version of derived invariance shows that $H_*(A,A) \iso H_*(A^{\mathit{perf}}, A^{\mathit{perf}})$. We refer to \cite{loday} for an introduction to Chern characters in noncommutative geometry, to \cite{keller98b} for derived invariance, and to \cite{shklyarov} for a more extensive discussion of the Hochschild classes associated to perfect modules.
\end{remark}

\subsection{The Cardy condition}
Let $Q$ be a $2n$-dimensional exact symplectic manifold which is convex at infinity. This is a common setup (see  \cite[Section 3a]{seidel-smith10}, to pick one of many occurrences). It means first of all that $Q$ carries a one-form $\theta_Q$ such that $\omega_Q = d\theta_Q$ is symplectic. Moreover, we have a compatible almost complex structure $J_Q$ with the property that $Q$ has an exhaustion by relatively compact subsets whose closures are $J_Q$-holomorphically convex. We will additionally assume that $Q$ is symplectically Calabi-Yau, hence admits a $J_Q$-complex volume form $\eta_Q$. Objects of the Fukaya category $\mathit{Fuk}(Q)$ are closed Lagrangian submanifolds $L \subset Q$ which are exact with respect to $\theta_Q$, and equipped with additional structures as before (a grading with respect to $\eta_Q$, a {\em Spin} structure, and a flat complex vector bundle).

The Fukaya category comes with a canonical {\em open-closed string map} from its Hochschild homology to the homology of the symplectic manifold,
\begin{equation} \label{eq:open-closed-homology}
\Phi: H_*(\mathit{Fuk}(Q),\mathit{Fuk}(Q)) \longrightarrow H^{*+n}_{\mathit{cpt}}(Q;\C) \iso H_{n-*}(Q;\C).
\end{equation}
The existence of this map, for the case of a closed symplectic manifold, was already implicit in \cite{kontsevich94}. Related ideas appear in various places in the Floer cohomology literature. We refer to \cite[Section 5]{abouzaid10} for a more extensive discussion (which is more sophisticated than what we need here, since it includes non-closed Lagrangian submanifolds and their wrapped Floer cohomology), and only give a schematic description. Fix a Morse function (and Riemannian metric) which can be used to define a Morse complex $C^*(Q)$ underlying ordinary cohomology. Suppose that we are given objects $(L_0,\dots,L_d)$ (assumed to come with trivial flat bundles for simplicity), generators $x_i \in \mathit{CF}^*(L_{i-1},L_i)$ (where we set $L_{-1} = L_d$) which correspond to (perturbed) intersection points, and a generator $y \in C^*(Q)$ which corresponds to a critical point. One then gets a number $n^{d+1}(y,x_0,\dots,x_d) \in \Z$ by counting (perturbed) pseudo-holomorphic maps which have asymptotics $x_i$ at the boundary marked points, and whose evaluation at the interior marked point lies on the unstable manifold of $y$ (Figure \ref{fig:disc}). These numbers form the coefficients of a map (not by itself a chain map)
\begin{equation}
C^*(Q) \otimes \mathit{CF}^*(L_d,L_0) \otimes \mathit{CF}^*(L_{d-1},L_d) \otimes \cdots \otimes \mathit{CF}^*(L_0,L_1) \longrightarrow \C[-n-d].
\end{equation}
The collection of such maps for all $d$ and $(L_0,\dots,L_d)$ forms the chain map underlying \eqref{eq:open-closed-homology} (up to an obvious dualization). 
\begin{figure}
\begin{centering}
\begin{picture}(0,0)%
\includegraphics{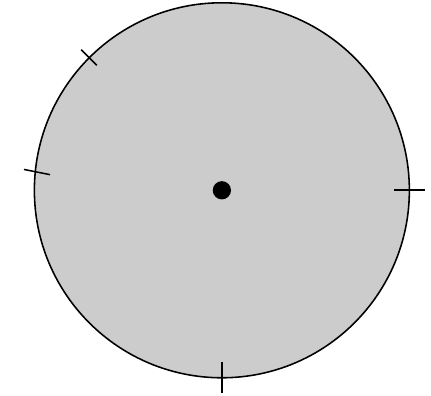}%
\end{picture}%
\setlength{\unitlength}{3947sp}%
\begingroup\makeatletter\ifx\SetFigFont\undefined%
\gdef\SetFigFont#1#2#3#4#5{%
  \reset@font\fontsize{#1}{#2pt}%
  \fontfamily{#3}\fontseries{#4}\fontshape{#5}%
  \selectfont}%
\fi\endgroup%
\begin{picture}(2052,1894)(1186,-448)
\put(2851,1289){\makebox(0,0)[lb]{\smash{{\SetFigFont{11}{13.2}{\rmdefault}{\mddefault}{\updefault}{\color[rgb]{0,0,0}$L_0$}%
}}}}
\put(2926,-286){\makebox(0,0)[lb]{\smash{{\SetFigFont{11}{13.2}{\rmdefault}{\mddefault}{\updefault}{\color[rgb]{0,0,0}$L_3$}%
}}}}
\put(1351,-136){\makebox(0,0)[lb]{\smash{{\SetFigFont{11}{13.2}{\rmdefault}{\mddefault}{\updefault}{\color[rgb]{0,0,0}$L_2$}%
}}}}
\put(1261,1064){\makebox(0,0)[lb]{\smash{{\SetFigFont{11}{13.2}{\rmdefault}{\mddefault}{\updefault}{\color[rgb]{0,0,0}$L_1$}%
}}}}
\end{picture}%
\caption{\label{fig:disc}}
\end{centering}
\end{figure}

By combining \eqref{eq:open-closed-homology} with \eqref{eq:end-to-hh} for a twisted complex $C$, one obtains a map
\begin{equation} \label{eq:pushforward}
\Phi_C: H^*(\mathit{hom}_{\mathit{Fuk}(Q)^{\mathit{tw}}}(C,C)) \longrightarrow H_{n-*}(Q;\C).
\end{equation}
The simplest example is when $C$ is a single Lagrangian submanifold $L$ equipped with a flat bundle $\rho$, in which case \eqref{eq:pushforward} recovers the classical (purely topological) map 
\begin{equation}
\mathit{HF}^*(L,L) \iso H^*(L;\mathit{End}(\rho)) \xrightarrow{\mathrm{trace}} H^*(L;\C) \iso H_{n-*}(L;\C) \longrightarrow H_{n-*}(Q;\C).
\end{equation}
In particular, the image of the identity $[e_L] \in \mathit{HF}^0(L,L)$ is $\mathrm{rank}(\rho)$ times the usual homology class $[L]$.
Similarly, given a twisted complex $C$, the image of $[e_C]$ can be computed from the homology classes of the Lagrangian submanifolds that enter into $C$. More generally, for any object $M$ of $\mathit{Fuk}(Q)^{\mathit{perf}}$, one can take the image of \eqref{eq:algebraic-fundamental-class} under \eqref{eq:open-closed-homology}, and thereby obtain a quasi-isomorphism invariant, which we write as
\begin{equation} \label{eq:geometric-fundamental-class}
[M] \in H_n(Q;\C).
\end{equation}
There is no a priori reason why \eqref{eq:geometric-fundamental-class} should be an integer class. However, one can obtain some restrictions on it from the following:

\begin{proposition} \label{th:cardy}
Suppose that we have two twisted complexes $C_k$ ($k = 0,1$), each of which comes with an endomorphism $[a_k] \in H^*(\mathit{hom}_{\mathit{Fuk}(Q)^{\mathit{tw}}}(C_k,C_k))$. Then
\begin{equation} \label{eq:cardy}
\Phi_{C_1}([a_1]) \cdot \Phi_{C_0}([a_0]) = (-1)^{n(n+1)/2}\, \mathrm{Str}\big([a] \longmapsto (-1)^{|a| \cdot |a_1|} [a_1] \cdot [a] \cdot [a_0]\big) \in \C.
\end{equation}
On the left side, we have the standard intersection pairing on homology. The right hand side is the supertrace of the endomorphism of $H^*(\mathit{hom}_{\mathit{Fuk}(Q)^{\mathit{tw}}}(C_0,C_1))$ given by composition with the two given $[a_k]$, with the additional sign as indicated (note that either side can be nonzero only if $|a_0|+|a_1| = 0$). \qed
\end{proposition}

This is a form of the Cardy condition, which a general feature of two-dimensional topological field theories involving both open strings (Floer cohomology for Lagrangian submanifolds, in our case) and closed strings (the homology of $Q$). For an occurrence in another context see for instance \cite[Theorem 15]{caldararu-willerton10}. Suppose first that $C_k = L_k$ are just Lagrangian submanifolds. In that situation, the gluing and deformation argument underlying \eqref{eq:cardy} becomes very simple; we summarize it in Figure \ref{fig:cardy0}. The general case is fundamentally parallel, except that the surfaces carry additional marked boundary points where one inserts the boundary operators of the twisted complexes. We do not carry out that general argument here; the intermediate level of generality, where one of the two objects involved is a Lagrangian submanifold and the other is a twisted complex, is discussed in detail in \cite{abouzaid10} (even though the intended application there is different). 
\begin{figure}
\begin{centering}
\begin{picture}(0,0)%
\includegraphics{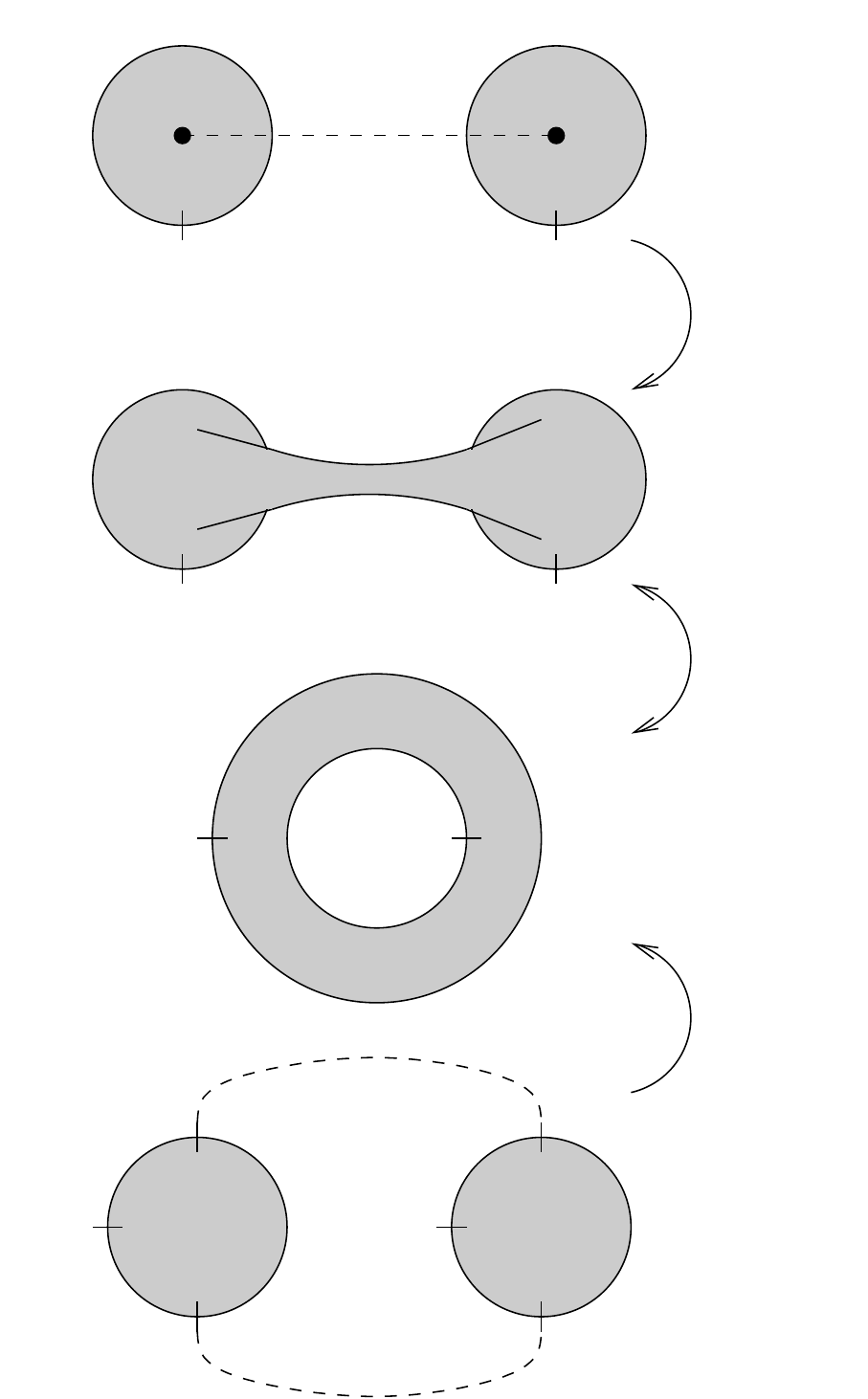}%
\end{picture}%
\setlength{\unitlength}{3947sp}%
\begingroup\makeatletter\ifx\SetFigFont\undefined%
\gdef\SetFigFont#1#2#3#4#5{%
  \reset@font\fontsize{#1}{#2pt}%
  \fontfamily{#3}\fontseries{#4}\fontshape{#5}%
  \selectfont}%
\fi\endgroup%
\begin{picture}(4339,7009)(286,-5723)
\put(2166,-4911){\makebox(0,0)[lb]{\smash{{\SetFigFont{11}{13.2}{\rmdefault}{\mddefault}{\updefault}{\color[rgb]{0,0,0}$[a_1]$}%
}}}}
\put(2426,-4561){\makebox(0,0)[lb]{\smash{{\SetFigFont{11}{13.2}{\rmdefault}{\mddefault}{\updefault}{\color[rgb]{0,0,0}$L_1$}%
}}}}
\put(3526,-4861){\makebox(0,0)[lb]{\smash{{\SetFigFont{11}{13.2}{\rmdefault}{\mddefault}{\updefault}{\color[rgb]{0,0,0}$L_0$}%
}}}}
\put(2426,-5236){\makebox(0,0)[lb]{\smash{{\SetFigFont{11}{13.2}{\rmdefault}{\mddefault}{\updefault}{\color[rgb]{0,0,0}$L_1$}%
}}}}
\put(3826,-361){\makebox(0,0)[lb]{\smash{{\SetFigFont{11}{13.2}{\rmdefault}{\mddefault}{\updefault}{\color[rgb]{0,0,0}gluing}%
}}}}
\put(3826,-2086){\makebox(0,0)[lb]{\smash{{\SetFigFont{11}{13.2}{\rmdefault}{\mddefault}{\updefault}{\color[rgb]{0,0,0}deformation}%
}}}}
\put(3826,-3886){\makebox(0,0)[lb]{\smash{{\SetFigFont{11}{13.2}{\rmdefault}{\mddefault}{\updefault}{\color[rgb]{0,0,0}gluing}%
}}}}
\put(1126,1139){\makebox(0,0)[lb]{\smash{{\SetFigFont{11}{13.2}{\rmdefault}{\mddefault}{\updefault}{\color[rgb]{0,0,0}$L_0$}%
}}}}
\put(3001,1139){\makebox(0,0)[lb]{\smash{{\SetFigFont{11}{13.2}{\rmdefault}{\mddefault}{\updefault}{\color[rgb]{0,0,0}$L_1$}%
}}}}
\put(1126,-586){\makebox(0,0)[lb]{\smash{{\SetFigFont{11}{13.2}{\rmdefault}{\mddefault}{\updefault}{\color[rgb]{0,0,0}$L_0$}%
}}}}
\put(3001,-586){\makebox(0,0)[lb]{\smash{{\SetFigFont{11}{13.2}{\rmdefault}{\mddefault}{\updefault}{\color[rgb]{0,0,0}$L_1$}%
}}}}
\put(1126,-61){\makebox(0,0)[lb]{\smash{{\SetFigFont{11}{13.2}{\rmdefault}{\mddefault}{\updefault}{\color[rgb]{0,0,0}$[a_0]$}%
}}}}
\put(3001,-61){\makebox(0,0)[lb]{\smash{{\SetFigFont{11}{13.2}{\rmdefault}{\mddefault}{\updefault}{\color[rgb]{0,0,0}$[a_1]$}%
}}}}
\put(1126,-1786){\makebox(0,0)[lb]{\smash{{\SetFigFont{11}{13.2}{\rmdefault}{\mddefault}{\updefault}{\color[rgb]{0,0,0}$[a_0]$}%
}}}}
\put(3001,-1786){\makebox(0,0)[lb]{\smash{{\SetFigFont{11}{13.2}{\rmdefault}{\mddefault}{\updefault}{\color[rgb]{0,0,0}$[a_1]$}%
}}}}
\put(2101,-2011){\makebox(0,0)[lb]{\smash{{\SetFigFont{11}{13.2}{\rmdefault}{\mddefault}{\updefault}{\color[rgb]{0,0,0}$L_0$}%
}}}}
\put(976,-2911){\makebox(0,0)[lb]{\smash{{\SetFigFont{11}{13.2}{\rmdefault}{\mddefault}{\updefault}{\color[rgb]{0,0,0}$[a_0]$}%
}}}}
\put(2251,-2986){\makebox(0,0)[lb]{\smash{{\SetFigFont{11}{13.2}{\rmdefault}{\mddefault}{\updefault}{\color[rgb]{0,0,0}$[a_1]$}%
}}}}
\put(2101,-2681){\makebox(0,0)[lb]{\smash{{\SetFigFont{11}{13.2}{\rmdefault}{\mddefault}{\updefault}{\color[rgb]{0,0,0}$L_1$}%
}}}}
\put(451,-4911){\makebox(0,0)[lb]{\smash{{\SetFigFont{11}{13.2}{\rmdefault}{\mddefault}{\updefault}{\color[rgb]{0,0,0}$[a_0]$}%
}}}}
\put(776,-4461){\makebox(0,0)[lb]{\smash{{\SetFigFont{11}{13.2}{\rmdefault}{\mddefault}{\updefault}{\color[rgb]{0,0,0}$L_0$}%
}}}}
\put(706,-5236){\makebox(0,0)[lb]{\smash{{\SetFigFont{11}{13.2}{\rmdefault}{\mddefault}{\updefault}{\color[rgb]{0,0,0}$L_0$}%
}}}}
\put(1651,-4561){\makebox(0,0)[lb]{\smash{{\SetFigFont{11}{13.2}{\rmdefault}{\mddefault}{\updefault}{\color[rgb]{0,0,0}$L_1$}%
}}}}
\end{picture}%
\caption{\label{fig:cardy0}}
\end{centering}
\end{figure}

\begin{remark}
We will not attempt to justify the signs in \eqref{eq:cardy}, but if the reader wants to check their plausibility, the following considerations might be useful. Suppose as before that the $C_k = L_k$ are Lagrangian submanifolds. If the $[a_k]$ are identity elements, we recover \eqref{eq:euler}. Next, if we switch $a_0$ and $a_1$ (restricting to the nontrivial case $|a_0| + |a_1| = 0$), the left hand side changes by $(-1)^{n+|a_0|}$. On the right hand side, for $(a_0,a_1)$ and $(a_1,a_0)$ we would be considering the sum of the traces of the maps
\begin{equation}
\begin{aligned}
& \mathit{HF}^k(L_0,L_1) \longrightarrow \mathit{HF}^k(L_0,L_1), && [a] \longmapsto (-1)^{k|a_1|+k} [a_1] \cdot [a] \cdot [a_0], \\
& \mathit{HF}^{n-k}(L_1,L_0) \longrightarrow \mathit{HF}^{n-k}(L_1,L_0), && [a'] \longmapsto (-1)^{(n-k)|a_0|+(n-k)} [a_0] \cdot [a'] \cdot [a_1].
\end{aligned}
\end{equation}
Under the ``Poincar{\'e} duality'' in Lagrangian Floer cohomology, these maps are dual up to a factor of $(-1)^{k|a_1| + k + (n-k)|a_0| + (n-k) + (n-|a_0|)|a_1|} = (-1)^{n+|a_0|}$.
\end{remark}

Returning to our discussion of \eqref{eq:geometric-fundamental-class}, Proposition \ref{th:cardy} implies that
\begin{equation} \label{eq:cardy2}
[M_1] \cdot [M_0] = (-1)^{n(n+1)/2} \chi(H^*(\mathit{hom}_{\mathit{Fuk}(Q)^{\mathit{perf}}}(M_0,M_1))) \in \Z,
\end{equation}
where $\chi$ is the Euler characteristic. In particular:


\begin{corollary} \label{th:homology-class}
Suppose that $H_n(Q;\Z)/\mathit{torsion}$ is generated by homology classes of objects in the Fukaya category. Let $L_0$ be an object of $\mathit{Fuk}(Q)$, carrying a flat vector bundle of rank $r_0$, and $M_1$ an object of $\mathit{Fuk}(Q)^{\mathit{perf}}$. Suppose that $H^*(\mathit{hom}_{\mathit{Fuk}(Q)^{\mathit{perf}}}(L_0,M_1))$ has odd total dimension. Then $r_0[L_0] \in H_n(Q;\Z)/\mathit{torsion}$ is not divisible by $2$ (and hence in particular nonzero).
\end{corollary}

\begin{proof} 
Let $V_1,\dots,V_m$ be a collection of objects in $\mathit{Fuk}(Q)$ whose homology classes generate $H_n(Q;\Z)/\mathit{torsion}$. From \eqref{eq:cardy2} we know that $[V_k] \cdot [M_1] \in \Z$ for all $k$. Arguing by contradiction, suppose that $r_0[L_0]$ is divisible by $2$, hence can be written as $r_0[L_0] = d_1[V_1] + \cdots d_r[V_r]$, where the $d_i$ are even integers. Then $r_0[L_0] \cdot [M_1] = d_1 [V_1] \cdot [M_1] + \cdots + d_r [V_r] \cdot [M_1]$ is also even, contradicting \eqref{eq:cardy2}.
\end{proof}

%
%
%

\subsection{The $(A_m)$ Milnor fibre revisited}
We now return to our example of $Q = Q_m^n$, for some $m \geq 1$ and $n \geq 2$.

\begin{lemma} \label{th:integral}
For any object $M$ of $\mathit{Fuk}(Q)^{\mathit{perf}}$, $[M] \in H_n(Q;\C)$ is integral.
\end{lemma}

\begin{proof}[Sketch of proof]
The simplest approach would be to enlarge $\mathit{Fuk}(Q)$ by allowing some noncompact Lagrangian submanifolds, and use an appropriate generalization of Proposition \ref{th:cardy}. Instead, one can use the following technically less demanding workaround. Let $\tilde{Q} = Q_{m+1}^n$, with its $(A_{m+1})$-configuration $(\tilde{V}_1,\dots,\tilde{V}_{m+1})$. We know that there is a symplectic embedding $Q \rightarrow \tilde{Q}$, which sends $V_k$ to $\tilde{V}_k$ for $k \leq m$. This induces a cohomologically full and faithful embedding $\mathit{Fuk}(Q) \rightarrow \mathit{Fuk}(\tilde{Q})$. The associated open-closed string maps fit into a commutative diagram
\begin{equation}
\xymatrix{
\ar[d] H^*(\mathit{Fuk}(Q),\mathit{Fuk}(Q)) \ar[rr] && H_{n-*}(Q;\C) \ar[d] \\
H^*(\mathit{Fuk}(\tilde{Q}),\mathit{Fuk}(\tilde{Q})) \ar[rr] && H_{n-*}(\tilde{Q};\C).
}
\end{equation}
Let $M$ be our object, and $\tilde{M}$ its image in $\mathit{Fuk}(\tilde{Q})^{\mathit{perf}}$. We then know that $[\tilde{M}]$ is the image of $[M]$ under the map $H_n(Q;\C) \rightarrow H_n(\tilde{Q};\C)$, which is just the inclusion $\Z^m = \Z^m \times \{0\} \hookrightarrow \Z^{m+1}$. Moreover, by Proposition \ref{th:cardy}, $[\tilde{M}] \cdot [\tilde{V}_k] \in \Z$ for all $k$ including $k = m+1$. It is easy to see, by inspection of the intersection forms \eqref{eq:intersection-form} and \eqref{eq:skew-form}, that this implies the integrality of $[M]$.
\end{proof}


The following is a weak version of Lemma \ref{th:algebraic-1}, whose proof is independent of the results of \cite{ishii-uehara05,ishii-ueda-uehara06}:

\begin{lemma} \label{th:algebraic-3}
Suppose that $m$ is even. Let $M_0$ be an object of $\mathit{Fuk}(Q)^{\mathit{perf}}$ whose cohomology level endomorphism ring is isomorphic to $H^*(S^n;\C)$. Then $H^*(\mathit{hom}_{\mathit{Fuk}(Q)^{\mathit{perf}}}(V_k,M_0))$ has odd total dimension for some $k$.
\end{lemma}

\begin{proof}[Sketch of proof]
As in our original proof of Lemma \ref{th:algebraic-1}, one can use general equivariance arguments to reduce the problem to the case $n = 2$, which we will exclusively consider from now on. From Lemma \ref{th:integral} we know that $[M_0] \in H_2(Q;\Z)$. Since $M_0$ is spherical, Proposition \ref{th:cardy} implies that $[M_0] \cdot [M_0] = -2$, hence $[M_0]$ is primitive and in particular nonzero mod $2$. Since $m$ is even, the intersection form \eqref{eq:intersection-form} has odd determinant, hence $[M_0] \cdot [V_k]$ is odd for some $k$. Applying Proposition \ref{th:cardy} again implies the desired result.
\end{proof}

This directly leads to a weaker version of Theorem \ref{th:primitive}, which states that if $L \subset Q$ is a rational homology sphere and {\em Spin}, then $[L]$ is not divisible by $2$ (and therefore in particular nonzero). There is also an analogous weak version of Theorem \ref{th:nonabelian}, which only excludes the existence of representations $\rho$ satisfying \eqref{eq:twisted-endomorphisms} whose rank $r$ is even. On the other hand, similar arguments yield the full strength of Theorems \ref{th:arnold} and \ref{th:abelian}. For that, one replaces the given proof of Lemma \ref{th:algebraic-2} by an argument based on Lemma \ref{th:integral} and Proposition \ref{th:cardy}. 

\section{More algebraic background}

\subsection{Group cohomology}
We now start the more abstract discussion of equivariance issues. We will only be considering the ground field $\C$, and the multiplicative group $G = \mathbb{G}_m = \C^*$. Let $\C[G] \iso \C[z,z^{-1}]$ be the ring of regular functions on $G$ as an affine algebraic variety, with its usual (pointwise) multiplication. This also carries a coproduct, coming from the group structure on $G$. {\em Rational representations} of $G$ are those which are direct sums of finite-dimensional representations. They can also be identified with comodules over the coalgebra $\C[G]$ (respecting the counit). While arbitrary direct sums of rational representations are rational, the same is not true for products.

Let $V$ be a rational representation of $G$. The bar resolution of $V$ is a chain complex of rational $G$-modules concentrated in degrees $\geq 0$,
\begin{equation} \label{eq:bar}
B^r(G,V) = \C[G]^{\otimes r+1} \otimes V,
\end{equation}
where $G$ acts only on the leftmost $\C[G]$ factor. To write down the differential, we prefer to think of elements of \eqref{eq:bar} as maps $b: G^{r+1} \rightarrow V$. Explicitly, $f_1 \otimes \cdots \otimes f_{r+1} \otimes v$ corresponds to $b(g_{r+1},\dots,g_1) = f_1(g_1) \dots f_{r+1}(g_{r+1}) v$. In these terms,
\begin{equation} \label{eq:delta-bar}
\begin{aligned}
& (\delta_{B^*(G,V)}b)(g_{r+1},\dots,g_1) = \sum_q (-1)^q b(g_{r+1},\dots,g_{q+1}g_q,\dots,g_1) \\ & \qquad \qquad \qquad \qquad \qquad \qquad + (-1)^{r+1} g_{r+1} b(g_r,\dots,g_1),
\end{aligned}
\end{equation}
which is equivariant for the $G$-action $(g \cdot b)(g_{r+1},\dots,g_1) = b(g_{r+1},\dots,g_1g)$. One needs to check that this preserves the subspace of functions of the form \eqref{eq:bar}, and that is done by rewriting \eqref{eq:delta-bar} in coalgebra terms (the same will be true in parallel situations later on).
%

\begin{lemma} \label{th:bar}
The comultiplication $V \rightarrow \C[G] \otimes V = B^0(G,V)$ induces an isomorphism
$V \iso H^*(B(G,V))$. \qed
\end{lemma}


Similarly, the group cochain complex with coefficients in $V$ is
\begin{equation} \label{eq:group-cocycle-equation}
\begin{aligned}
& C^r(G,V) = \C[G]^{\otimes r} \otimes V, \\
& (\delta_{C^*(G,V)}c)(g_r,\dots,g_1) = \sum_q (-1)^q c(g_r,\dots,g_{q+1}g_q,\dots,g_1) \\ & \qquad \qquad \qquad \qquad + (-1)^r g_r c(g_{r-1},\dots,g_1) + c(g_r,\dots,g_2).
\end{aligned}
\end{equation}

\begin{lemma} \label{th:cohomology}
The inclusion $V^G \hookrightarrow V = C^0(G,V)$ induces an isomorphism $V^G \iso H^*(C(G,V))$. \qed
\end{lemma}

While Lemma \ref{th:bar} is a general fact about linear algebraic groups, Lemma \ref{th:cohomology} relies on the semisimplicity of $G$. Of course, both results are classical, see \cite{hochschild61}.

\subsection{Families of modules}
We return to the discussion of $A_\infty$-structures from Section \ref{subsec:ainfty}, in order to mention a less familiar aspect, namely the theory of {\em families of $A_\infty$-modules} \cite[Section 1]{seidel11}. Let $A$ be a strictly unital $A_\infty$-algebra. Take a smooth affine variety $X$, with coordinate ring $\C[X]$. A family of $A$-modules over $X$ is given by a graded $\C[X]$-module $N$ such that each $N^k$ is projective, together with structure maps
\begin{equation} \label{eq:family-mu}
\mu_N^{d+1}: N \otimes_\C A^{\otimes d} \longrightarrow N[1-d]
\end{equation}
which are $\C[X]$-linear and satisfy the usual conditions. It may be useful to recall that over rings with finite global dimension, such as $\C[X]$, unbounded complexes of projective modules are well-behaved \cite{dold60}. 
Families of modules over $X$ form a dg category linear over $\C[X]$, which we denote by $A^{\mathit{mod}}/X$. 

The simplest example is the trivial (or constant) family associated to an $A$-module $M$, which is $\C[X] \otimes M$ with the $A_\infty$-module structure of $M$ extended $\C[X]$-linearly.  Take an $A_\infty$-module $M$ and an arbitrary family of modules $N$. We then have
\begin{equation} \label{eq:adjoint-family}
\mathit{hom}_{A^{\mathit{mod}}/X}(\C[X] \otimes M,N) \iso \mathit{hom}_{A^{\mathit{mod}}}(M,N).
\end{equation}
On the right hand side, we forget the $\C[X]$-module structure of $N$ (so this forgetting is adjoint to the operation of constructing a constant family). In the special case of two constant families, we have a natural map
\begin{equation} \label{eq:constant-hom}
\C[X] \otimes \mathit{hom}_{A^{\mathit{mod}}}(M_0,M_1) \longrightarrow \mathit{hom}_{A^{\mathit{mod}}/X}(\C[X] \otimes M_0,\C[X] \otimes M_1). 
\end{equation}
If $M_0$ is perfect, one can use \eqref{eq:adjoint-family} and \eqref{eq:compact} to conclude that this map is a quasi-isomorphism. Returning to the general situation, a family of $A_\infty$-modules is called {\em perfect} if, up to quasi-isomorphism, it can be constructed from the trivial family $\C[X] \otimes A$ by a finite sequence of the usual operations (shifts, mapping cones, and finally passing to a direct summand). In particular, if $M$ is a perfect module then $\C[X] \otimes M$ is a perfect family. 

\begin{lemma} \label{th:finite-generation}
Suppose that $N_0$ is a perfect family, and that $N_1$ is a family such that $H^*(N_1)$ is a finitely generated $\C[X]$-module in each degree. Then $H^*(\mathit{hom}_{A^{\mathit{mod}}/X}(N_0,N_1))$ is again finitely generated over $\C[X]$ in each degree.
\end{lemma}

\begin{proof}
This is a tautology if $N_0 = \C[X] \otimes A$, and the general case follows from that by going through suitable long exact sequences.
\end{proof}

We will also need to recall the definition of a {\em connection} on a family $N$ \cite[Section 1h]{seidel11}. We first introduce the more general notion of pre-connection. Such a pre-connection is a sequence of maps
\begin{equation} \label{eq:nabla}
\pabla_N^{d+1}: N \otimes_\C A^{\otimes d} \longrightarrow \Omega_X \otimes_{\C[X]} N[-d],
\end{equation}
where $\Omega_X$ is the module of K{\"a}hler differentials. The first term should be of the form $\pabla^1_N(n) = (-1)^{|n|} D_N n$, where $D_N$ is an ordinary connection on the graded $\C[X]$-module $N$. The higher order terms $\pabla_N^{d+1}$, $d>0$, are $\C[X]$-linear. The failure of this to be compatible with the $A_\infty$-module structure is expressed by the {\em deformation cocycle}
\begin{equation} \label{eq:atiyah}
\mathit{def}_N = \mu^1_{A^{\mathit{mod}}/X}(\pabla_N) \in \mathit{hom}_{A^{\mathit{mod}}/X}^1(N,\Omega_X \otimes_{\C[X]} N).
\end{equation}
Here, we apply the usual formula for $\mu^1_{A^{\mathit{mod}}/X}$, which is the differential on morphisms in the category of (families of) $A_\infty$-modules, to $\pabla_N$, irrespectively of the fact that $\pabla_N^1$ is not $\C[X]$-linear. Explicitly, we follow the sign conventions in \cite[Section 1j]{seidel04}, so 
\begin{equation}
\begin{aligned}
& \mathit{def}_N^1(n) = (\mathit{id}_{\Omega_X} \otimes \mu^1_N)(D_Nn) - D_N(\mu^1_N(n)), \\
& \mathit{def}_N^2(n,a) = (-1)^{|n|+|a|-1} (\mathit{id}_{\Omega_X} \otimes \mu_N^1)(\pabla_N^2(n,a)) 
+ (\mathit{id}_{\Omega_X} \otimes \mu_N^2)(D_Nn,a) \\ 
& \qquad \qquad - D_N \mu_N^2(n,a) + (-1)^{|n|} \pabla_N^2(\mu_N^1(n),a) + (-1)^{|n|+|a|-1} \pabla_N^2(n,\mu_A^1(a)), \\
& \dots
\end{aligned}
\end{equation}
Because of cancellation between the terms involving $\pabla_N^1$, $\mathit{def}_N^d$ is indeed $\C[X]$-linear in $n$ (recall that $a$ is an element of $A$, hence is constant over $X$). By construction, $\mu^1_{A^{\mathit{mod}}/X}(\mathit{def}_N) = 0$, and the class
\begin{equation}
\mathit{Def}_N = [\mathit{def}_N] \in H^1(\mathit{hom}_{A^{\mathit{mod}}/X}(N,\Omega_X \otimes_{\C[X]} N))
\end{equation}
is independent of the choice of $\pabla_N$. In particular, to compute that class one may choose a pre-connection whose higher order terms vanish, in which case \eqref{eq:atiyah} simplifies to
\begin{equation}
\mathit{def}_N^{d+1}(n,a_d,\dots,a_1) = 
(\mathit{id}_{\Omega_X} \otimes \mu_N^{d+1})(D_N n,a_d,\dots,a_1) - D_N \mu_N^{d+1}(n,a_d,\dots,a_1),
\end{equation}
measuring the compatibility of $D_N$ with the $A_\infty$-module structure. A connection $\nabla_N$ is a pre-connection such that $\mathit{def}_N = 0$. Connections exist if and only if $\mathit{Def}_N = 0$; and in that case, homotopy classes of connections are parametrized by $H^0(\mathit{hom}_{A^{\mathit{mod}}/X}(N,\Omega_X \otimes_{\C[X]} N))$. If $N_0$ and $N_1$ are families equipped with connections, then the graded $\C[X]$-module $H^*(\mathit{hom}_{A^{\mathit{mod}}/X}(N_0,N_1))$ carries a connection in the ordinary sense of the word, denoted by $D_{N_0,N_1}$ (see again \cite[Section 1h]{seidel11} for the definition). In particular, if this module is finitely generated (in some degree), it is necessarily projective (in that degree) \cite{bernstein-dmodules}. Moreover, these connections are compatible with the categorical structure in the following sense: if the families $N_k$ all carry connections, then
\begin{equation} \label{eq:product-of-connections}
\begin{aligned}
& D_{N_0,N_2}([n_2] \, [n_1]) = D_{N_1,N_2}([n_2]) [n_1] + (\mathit{id}_{\Omega_X} \otimes [n_2]) D_{N_0,N_1}([n_1]), \\
& \text{for } [n_k] \in H^*(\mathit{hom}_{A^{\mathit{mod}}/X}(N_{k-1},N_k)).
\end{aligned}
\end{equation}

\section{Weakly equivariant modules}

\subsection{Setup}
From now on, suppose that our $A_\infty$-algebra $A$ carries a rational action of $G = \C^*$. This should be understood in a naive sense: we have a rational action of $G$ on the underlying graded vector space, such that all the $\mu^d_A$ are equivariant. 

Given an $A_\infty$-module $M$ and some $h \in G$, one can define the pullback module $h^*M$, which has the same underlying graded vector space as $M$, but with the twisted $A_\infty$-module structure
\begin{equation}
\mu_{h^*M}^{d+1}(m,a_d,\dots,a_1) = \mu^{d+1}_M(m,h(a_d),\dots,h(a_1)).
\end{equation}
Pullback by $h$ is an automorphism of the category of $A$-modules, whose action on morphisms is
\begin{equation} \label{eq:pullback}
\begin{aligned}
& h^*: \mathit{hom}_{A^{\mathit{mod}}}(M_0,M_1) \longrightarrow \mathit{hom}_{A^{\mathit{mod}}}(h^*M_0,h^*M_1), \\
& (h^*\phi)^{d+1}(m,a_d,\dots,a_1) = \phi^{d+1}(m,h(a_d),\dots,h(a_1)).
\end{aligned}
\end{equation}
There is also an infinitesimal version of pullback. Namely, define the {\em Killing cocycle} 
\begin{equation} \label{eq:ob}
\begin{aligned}
& \mathit{ki}_M \in \mathit{hom}_{A^{\mathit{mod}}}^1(M,\g^* \otimes M), \\
& \mathit{ki}_M^{d+1}(m,a_d,\dots,a_1) = -\sum_k z^* \otimes \mu_M^{d+1}(m,a_d,\dots, z(a_k),\dots,a_1).
\end{aligned}
\end{equation}
Here, $\g = \C$ is the Lie algebra of $G = \C^*$. We picked a nonzero element $z \in \g$ along with its dual $z^* \in \g^*$. The notation $z(a_k)$ stands for the infinitesimal action of $\g$ on $A$. The cohomology class $\mathit{Ki}_M = [\mathit{ki}_M] \in H^1(\mathit{hom}_{A^{\mathit{mod}}}(M,\g^* \otimes M))$ is a quasi-isomorphism invariant. As will be explained below, this is the infinitesimal obstruction to equivariance.

Because of the structure of $G$ as an algebraic group, it is often better to think in terms of families parametrized by $G$. The category $A^{\mathit{mod}}/G$ has a $\C[G]$-linear automorphism $\gamma^*$, which acts by $g^*$ on the fibre over $g \in G$. We start with the trivial family $\C[G] \otimes M$, and define the {\em orbit family} to be $N = \gamma^*(\C[G] \otimes M)$. The graded $\C[G]$-module underlying $N$ is still $\C[G] \otimes M$, but the $A_\infty$-module structure has been modified in such a way that the fibre of $N$ over any point $g$ is $g^*M$. In the special case where $M = A$ is the free module, one can use the $G$-action to trivialize the associated orbit family, meaning that $\gamma^*(\C[G] \otimes A) \iso \C[G] \otimes A$. From this it follows that, for any perfect module $M$, the associated orbit family $N$ is a perfect family.

We will now relate this construction to the previously introduced infinitesimal obstruction. By combining \eqref{eq:adjoint-family} and the pullback $\gamma^*$, we get an isomorphism of chain complexes
\begin{equation} \label{eq:gamma-endo}
\begin{aligned}
\mathit{hom}_{A^{\mathit{mod}}}(M,\C[G] 
\otimes \g^* \otimes M) & = \mathit{hom}_{A^{\mathit{mod}}/G}(\C[G] \otimes M, \Omega_G \otimes M) \\ & \xrightarrow{\gamma^*} \mathit{hom}_{A^{\mathit{mod}}/G}(N, \Omega_G \otimes_{\C[G]} N).
\end{aligned}
\end{equation}
A simple computation shows that the image of $\mathit{ki}_M \in \mathit{hom}_{A^{\mathit{mod}}}(M,\g^* \otimes M)$ under \eqref{eq:gamma-endo} is the deformation class associated to the trivial pre-connection $\pabla_N$ (the one obtained by identifying the underlying $\C[G]$-module with $\C[G] \otimes M$). Since the inclusion of constants, $\mathit{hom}_{A^{\mathit{mod}}}(M,\g^* \otimes M) \hookrightarrow \mathit{hom}_{A^{\mathit{mod}}}(M,\C[G] \otimes \g^* \otimes M)$, is injective on cohomology, $N$ admits a connection if and only if $\mathit{Ki}_M = 0$. More concretely, a choice of coboundary $\mu^1_{A^{\mathit{mod}}}(\alpha) = \mathit{ki}_M$ yields a connection $\nabla_N$ on $N$, given by
\begin{equation} \label{eq:symmetric-connection}
\begin{aligned}
& \nabla_N^1(n)(g) = (-1)^{|n|} dn(g) - \alpha^1(n(g)), \\
& \nabla_N^{d+1} (n,a_d,\dots,a_1)(g) = -\alpha^{d+1}(n(g),g(a_d),\dots,g(a_1)) \quad \text{for $d>0$.}
\end{aligned}
\end{equation}
Here, we think of $n \in N \iso \C[G] \otimes M$ as a function on $G$ with values in $M$; and of $\nabla n \in \Omega_G \otimes_{\C[G]} N \iso \Omega_G \otimes M \iso \C[G] \otimes \g^* \otimes M$ as a function on $G$ with values in $\g^* \otimes M$. 

\begin{remark} \label{th:2-pullbacks}
The orbit family has a symmetry property, which is intuitively obvious but whose formal statement requires a bit of effort. Let's temporarily suppose that $N$ is an arbitrary family of $A$-modules over $G$. Given some fixed $h \in G$, we can form $h^*N$, which is the pullback \eqref{eq:pullback} applied equally to all the fibres. On the other hand, we can use the right multiplication map $r_{h^{-1}}: G \rightarrow G$, $r_{h^{-1}}(g) = gh^{-1}$, and push forward the family $N$ in the geometric sense, forming $r_{h^{-1},*} N$. The reader will have noticed the unhappy proximity in notation between these two quite different operations ($h^*$ changes the $A_\infty$-module structure, whereas $r_{h^{-1},*}$ changes the $\C[G]$-module structure). To clarify the situation, we write things down explicitly: the graded vector space underlying $h^*N$ and $r_{h^{-1},*}N$ is equal to that of $N$, and the $\C[G]$-module structure and $A_\infty$-module structure are given by
\begin{equation} \label{eq:two-pullbacks}
\begin{aligned}
& (f \cdot_{h^*N} n)(g) = f(g) n(g), && \mu_{h^*N}^{d+1}(n,a_d,\dots,a_1) = \mu^{d+1}_N(n,h(a_d),\dots,h(a_1)), \\
& (f \cdot_{r_{h^{-1},*} N} n)(g) = f(gh^{-1})n(g), && \mu_{r_{h^{-1},*}N}^{d+1}(n,a_d,\dots,a_1) = \mu^{d+1}_N(n,a_d,\dots,a_1).
\end{aligned}
\end{equation}
In the case of the orbit family, it follows that $n(g) \mapsto n(gh)$ is an isomorphism (of $\C[G]$-modules, and strictly compatible with the $A_\infty$-module structure) $r_{h^{-1},*}N \rightarrow h^*N$. Note also that any connection on $N$ induces ones on $h^*N$ and $r_{h^{-1},*}N$, which are different in general. However, for \eqref{eq:symmetric-connection} those two connections are related by the isomorphism which we have just introduced.
\end{remark}

\subsection{Naive actions}
The notion of action on a module directly corresponding to the one we've adopted for $A_\infty$-algebras is:

\begin{definition}
A {\em naive $G$-action} on an $A$-module $M$ is a rational action of $G$ on the underlying graded vector space, such that all the $\mu^{d+1}_M$ are equivariant.
\end{definition}

\begin{example} \label{th:homogeneous}
Suppose that $A$ comes from a graded algebra (only $\mu^2_A$ is nonzero). Then, it carries a $G$-action which has weight $i$ precisely on the degree $i$ part. Suppose that $M$ carries a naive $G$-action, and write $M^{i,j}$ for the part of $M^{i+j}$ on which the action has weight $j$. Then, the bigraded space $M$ with its operations $\mu_M^{d+1}$ is precisely an object of the category $A^{\Bmod}$ considered in Section \ref{subsec:ainfty}.
\end{example}

If $M$ carries a naive $G$-action, then $\mathit{Ki}_M$ is trivial, since the cocycle $\mathit{ki}_M$ is the coboundary of the linear endomorphism $m \mapsto (-1)^{|m|} z^* \otimes z(m)$ (in the contrapositive, if $\mathit{Ki}_M$ is nonzero, not only does $M$ not carry a naive $G$-action, but neither can any other quasi-isomorphic module). From a more geometric viewpoint, a naive action gives rise to an isomorphism $\C[G] \otimes M \iso N$, which takes $m(g)$ to $n = g m(g)$. The existence of such an isomorphism implies that the deformation class of $N$ must vanish, which as explained before is equivalent to the vanishing of $\mathit{Ki}_M$. 

In spite of the apparently obvious nature of the definition, there are some points of caution. For instance, if $M_0$ and $M_1$ carry naive $G$-actions, the space $\mathit{hom}_{A^{\mathit{mod}}}(M_0,M_1)$ carries a $G$-action, but that may not be rational in general (because the definition involves a direct product). 

\begin{lemma} \label{th:rational-hom}
Suppose that $M_0$ and $M_1$ carry naive $G$-actions, and that $M_0$ is equivariantly perfect. This means that in the category of modules with naive $G$-actions (and equivariant maps), $M_0$ is quasi-isomorphic to a direct summand of an object produced from $A$ by the following operations: changing the group actions by tensoring with a character; shifting the grading; and mapping cones. Then $\mathit{hom}_{A^{\mathit{mod}}}(M_0,M_1)$ is $G$-equivariantly quasi-isomorphic to a chain complex of rational $G$-modules.
\end{lemma}

\begin{proof}
This is elementary if $M_0$ can be constructed as an equivariant twisted complex, which means without taking a direct summand. The general case follows from the fact that the derived category of rational $G$-modules admits countable direct sums, hence is closed under homotopy retracts \cite[Proposition 2.2]{lueck-ranicki87}.
\end{proof}

%
%

\subsection{Strict actions}
Instead of asking for $G$ to act on $M$ linearly, one can allow higher order terms. What one then wants is, for each $g \in G$, a homomorphism
\begin{equation}
\rho^1(g) \in \mathit{hom}^0_{A^{\mathit{mod}}}(M,g^*M).
\end{equation}
There is a {\em unitality} condition, which says that this should be the identity if $g = e$. We also need a {\em cocycle} condition which is formulated in terms of pullbacks, namely
\begin{equation} \label{eq:1-cocycle-condition}
\mu^2_{A^{\mathit{mod}}}(g_1^*\rho^1(g_2),\rho^1(g_1)) = \rho^1(g_2g_1) \in \mathit{hom}^0_{A^{\mathit{mod}}}(M,g_1^*g_2^*M).
\end{equation}

It may be worth while to spell this out a little. The condition that $\rho^1(g)$ should be a module homomorphism, meaning that $\mu^1_{A^{\mathit{mod}}}(\rho^1(g)) = 0$, yields an infinite sequence of equations
\begin{equation} \label{eq:cocycle1}
\begin{aligned}
 & \sum_i (-1)^{|a_{i+1}| + \cdots + |a_d| + |m| + d-i} \mu_M^{i+1}(\rho^{1,d+1-i}(g,m,a_d,\dots,a_{i+1}),g(a_i),\dots,g(a_1)) \\
 + & \sum_i (-1)^{|a_{i+1}| + \cdots + |a_d| + |m| + d-i} \rho^{1,i+1}(g,\mu_M^{d-i+1}(m,a_d,\dots,a_{i+1}),a_i,\dots,a_1) \\
 + & \sum_{i,j} (-1)^{|a_{i+1}| + \cdots + |a_d| + |m| + d-i} \rho^{1,d+2-j}(g,m,a_d,\dots,a_{i+j+1}, \\[-1em]
& \qquad \qquad \qquad \qquad \qquad \qquad \qquad \qquad \qquad \qquad \mu_A^j(a_{i+j},\dots,a_{i+1}),a_i,\dots,a_1) =  0,
\end{aligned}
\end{equation}
of which the two simplest ones are
\begin{align} 
\label{eq:l2} & \mu_M^1(\rho^{1,1}(g,m)) + \rho^{1,1}(g,\mu_M^1(m)) = 0, \\
\label{eq:l4}
 & \mu_M^1(\rho^{1,2}(g,m,a)) + (-1)^{|a|-1} \rho^{1,2}(g,\mu_M^1(m),a) + \rho^{1,2}(g,m,\mu^1_A(a)) \\ \notag &
 \qquad \qquad \qquad \qquad +
 (-1)^{|a|-1} \mu_M^2(\rho^{1,1}(g,m),g(a)) + \rho^{1,1}(g,\mu_M^2(m,a)) = 0.
\end{align}
\eqref{eq:l2} implies that for each $g$, the map
\begin{equation} \label{eq:strict-map}
m \longmapsto (-1)^{|m|} \rho^{1,1}(g,m)
\end{equation}
is an endomorphism of $M$ as a chain complex. \eqref{eq:l4} says that the map on cohomology induced by \eqref{eq:strict-map} is a homomorphism from the $H(A)$-module $H(M)$ to the pullback module $g^*H(M)$. The other requirement is \eqref{eq:1-cocycle-condition}, which yields
\begin{equation} \label{eq:cocycle2}
\begin{aligned}
& \sum_i (-1)^{|a_{i+1}| + \cdots + |a_d| + |m| + d-i}\rho^{1,i+1}(g_2,\rho^{1,d+1-i}(g_1,m,a_d,\dots,a_{i+1}),g_1(a_i),\dots,g_1(a_1))
 \\[-1em]
& \qquad \qquad \qquad \qquad \qquad \qquad \qquad \qquad \qquad \qquad \qquad - \; \rho^{1,d+1}(g_2g_1,m,a_d,\dots,a_1) = 0.
\end{aligned}
\end{equation}
The simplest of these equations is
\begin{equation} \label{eq:strict-composition}
(-1)^{|m|}\rho^{1,1}(g_2,\rho^{1,1}(g_1,m)) - \rho^{1,1}(g_2g_1,m) = 0,
\end{equation}
which says that the maps \eqref{eq:strict-map} are strictly compatible with the group structure of $G$. On the cohomology level, the outcome is that $H(M)$ is an equivariant $H(A)$-module in the classical sense.

So far, we have treated $G$ as a discrete group. To take its algebraic nature into account, we further impose a {\em rationality} condition, namely that the maps $\rho^{1,d+1}(g)$ for varying $g$ should be specializations of a single linear map
\begin{equation} \label{eq:phi-maps-1}
\rho^{1,d+1}: M \otimes A^{\otimes d} \longrightarrow \C[G] \otimes M[-d].
\end{equation}
Rationality becomes more natural from a geometric viewpoint. Spelling out \eqref{eq:adjoint-family} yields
\begin{equation} \label{eq:r-equals-1}
\mathit{hom}_{A^{\mathit{mod}}/G}(\C[G] \otimes M,N) \iso \prod_d \mathit{Hom}_\C(M \otimes A[1]^{\otimes d}, \C[G] \otimes M),
\end{equation}
hence can think of \eqref{eq:phi-maps-1} as a single element
\begin{equation} \label{eq:rho-1}
\rho^1 \in \mathit{hom}_{A^{\mathit{mod}}/G}(\C[G] \otimes M,N),
\end{equation}
which then satisfies $\mu^1_{A^{\mathit{mod}}/G}(\rho^1) = 0$.

To formulate the cocycle condition in a similar way, it is convenient to introduce some higher-dimensional families related to $N$ (even if that makes the notation somewhat more cumbersome). Let $G^r$ be the product of $r$ copies of $G$, so $\C[G^r] = \C[G]^{\otimes r}$. These products come with projection and multiplication maps
\begin{equation}
\begin{aligned}
& p_{j,i}: G^r \longrightarrow G^{r-j+i}, \quad
p_{j,i}(g_r,\dots,g_1) = (g_r,\dots,g_{j+1},g_{i-1},\dots,g_1), \\
& m_{j,i}: G^r \longrightarrow G^{r-j+i+1}, \quad
m_{j,i}(g_r,\dots,g_1) = (g_r,\dots,g_jg_{j-1} \cdots g_{i+1}g_i,g_{i-1},\dots,g_1).
\end{aligned}
\end{equation}
Let $N^r$ be the pullback of the orbit family by the total multiplication map $m_{r,1}: G^r \rightarrow G$ (this includes the case $r = 1$, where $N^1 = N$; and the degenerate case $r = 0$, where $N^0 = M$ considered as a family over a point). Equivalently, let $\gamma_k^*$ be the automorphism of $A^{\mathit{mod}}/G^r$ which acts by $g_k^*$ on the fibre over $(g_r,\dots,g_1)$. Then $N^r$ is the image of the trivial family $p_{r,1}^*N^0 = \C[G^r] \otimes M$ under $\gamma_1^*\cdots \gamma_r^*$. Generalizing the $r = 1$ case from \eqref{eq:r-equals-1}, one has an isomorphism of graded $\C[G^r]$-modules
\begin{equation} \label{eq:family-hom}
\mathit{hom}_{A^{\mathit{mod}}/G^r}(p_{r,1}^* N^0,N^r) 
\iso \prod_d \mathit{Hom}_\C(M \otimes A[1]^{\otimes d}, \C[G]^{\otimes r} \otimes M),
\end{equation}
and can then rewrite \eqref{eq:1-cocycle-condition} as an equation lying in the $r = 2$ case of that space,
\begin{equation} \label{eq:uniform-1-cocycle}
m_{2,1}^*\rho^1 = \mu^2_{A^{\mathit{mod}}/G^2}(\gamma_1^* p_{1,1}^* \rho^1, p_{2,2}^* \rho^1) \in
\mathit{hom}_{A^{\mathit{mod}}/G^2}(p_{2,1}^*N^0,N^2).
\end{equation}
The $m^*$ and $p^*$ are geometric pullbacks, which change the parameter space of a family (hence the underlying graded vector space), while $\gamma_1^*$ affects only the $A_\infty$-module structure. Let's summarize the discussion so far:

\begin{definition}
A {\em strict $G$-action} on an $A$-module $M$ is given by a family of maps $\rho^1(g)$ satisfying the appropriate unitality, cocycle and rationality conditions. Equivalently, it is given by a homomorphism of families of modules $\rho^1$ as in \eqref{eq:rho-1}, whose restriction to the fibre at $e$ is the identity map, and such that \eqref{eq:uniform-1-cocycle} holds.
\end{definition}

Given a naive $G$-action on $M$, one can define a strict $G$-action by setting $\rho^{1,1}(g,m) = (-1)^{|m|} gm$ and $\rho^{1,d+1}(g,m,a_d,\dots,a_1) = 0$ for $d>0$. In converse direction one has:

\begin{lemma}
A strict group action on $M$ induces a naive group action on the quasi-isomorphic module $M \otimes_A A$. \qed
\end{lemma}

This is quite straightforward to prove: the $\rho^1(g)$ and the $G$-action on $A$ induce linear maps $g^*(M \otimes_A A) \rightarrow M \otimes_A A$, and these satisfy all the conditions for a naive group action. Unfortunately, the outcome of this discussion is that strict actions are not fundamentally more general than naive actions. Indeed, we have mentioned them mainly because they represent a natural stepping-stone towards another notion, which we will introduce next.

\subsection{Weak actions}
The previous discussion suggests another, more substantial, generalization of the notion of group action on an $A_\infty$-module, which is to require all the relevant conditions to hold only on the level of the cohomological category $H^0(A^{\mathit{mod}})$. More precisely:

\begin{definition}
A {\em weak $G$-action} on $M$ is given by a homomorphism \eqref{eq:rho-1} of families of $A_\infty$-modules (assumed to be closed with respect to $\mu^1_{A^{\mathit{mod}}/G}$ as before), whose restriction to the fibre at $e \in G$ represents the identity in $H^0(\mathit{hom}_{A^{\mathit{mod}}}(M,M))$, and such that the equality \eqref{eq:uniform-1-cocycle} holds in $H^0(\mathit{hom}_{A^{\mathit{mod}}/G^2}(p_{2,1}^*N^0,N^2))$.
\end{definition}

This means that there is another map $\rho^2 \in \mathit{hom}_{A^{\mathit{mod}}/G^2}^{-1}(p_{2,1}^*N^0,N^2)$ such that
\begin{equation} \label{eq:2-cocycle-short}
\mu^1_{A^{\mathit{mod}}/G^2}(\rho^2) \\
+ \mu^2_{A^{\mathit{mod}}/G^2}(\gamma_1^* p_{1,1}^* \rho^1, p_{2,2}^* \rho^1) \\
- m_{2,1}^*\rho^1 = 0.
\end{equation}
When written out explicitly in terms of \eqref{eq:family-hom} for $r = 2$, the components of $\rho^2$ are maps
\begin{equation}
\rho^{2,d+1}: M \otimes A^{\otimes d} \longrightarrow \C[G]^{\otimes 2} \otimes M[-d-1],
\end{equation}
and \eqref{eq:2-cocycle-short} says that
\begin{equation} \label{eq:2-cocycle}
\begin{aligned}
&  \sum_i (-1)^{|a_{i+1}| + \cdots + |a_d| + |m| + d-i} \mu_M^{i+1}(\rho^{2,d+1-i}(g_2,g_1,m,a_d,\dots,a_{i+1}),g_2g_1(a_i),\dots,g_2g_1(a_1)) \\
 + & \sum_i (-1)^{|a_{i+1}| + \cdots + |a_d| + |m| + d-i} \rho^{2,i+1}(g_2,g_1,\mu_M^{d-i+1}(m,a_d,\dots,a_{i+1}),a_i,\dots,a_1) \\
 + & \sum_{i,j} (-1)^{|a_{i+1}| + \cdots + |a_d| + |m| + d-i} \rho^{2,d+2-j}(g_2,g_1,m,a_d,\dots,a_{i+j+1}, \mu_A^j(a_{i+j},\dots,a_{i+1}),a_i,\dots,a_1) \\ 
+ & \sum_i (-1)^{|a_{i+1}| + \cdots + |a_d| + |m| + d-i} \rho^{1,i+1}(g_2,\rho^{1,d+1-i}(g_1,m,a_d,\dots,a_{i+1}),g_1(a_i),\dots,g_1(a_1))
 \\
- & \rho^{1,d+1}(g_2g_1,m,a_d,\dots,a_1) = 0. \\
\end{aligned}
\end{equation}
The simplest of these equations is
\begin{equation}
\begin{aligned}
\label{eq:l3}
 & \mu_M^1(\rho^{2,1}(g_2,g_1,m)) + \rho^{2,1}(g_2,g_1,\mu_M^1(m)) \\ & \notag \qquad \qquad \qquad \qquad +
 \rho^{1,1}(g_2,\rho^{1,1}(g_1,m)) - \rho^{1,1}(g_2g_1,m) = 0.
\end{aligned}
\end{equation}
This means that the chain maps \eqref{eq:strict-map} are compatible with the group structure of $G$ only up to chain homotopies, which are given by $\rho^{2,1}$. Of course, on the level of cohomology the outcome is still that $H(M)$ is an equivariant $H(A)$-module in the standard sense.

We need to further discuss the implications of having a weak $G$-action, since these are maybe not as obvious as for the previously discussed notions of group action. Specializing to the fibre over $g \in G$ yields $\rho^1(g) \in \mathit{hom}_{A^{\mathit{mod}}}^0(M,g^*M)$, and these satisfy the analogue of \eqref{eq:1-cocycle-condition} on the cohomology level. In particular, each $\rho^1(g)$ is an isomorphism in the category $H^0(A^{\mathit{mod}})$. As a consequence, we get an induced $G$-action (in the ordinary sense of linear action) on the space $H^*(\mathit{hom}_{A^{\mathit{mod}}}(M,M))$. This is defined by filling in the diagram (up to chain homotopy)
\begin{equation} \label{eq:endo-representation}
\xymatrix{
\mathit{hom}_{A^{\mathit{mod}}}(M,M) \ar[rrr]^{\mu^2_{A^{\mathit{mod}}}(\rho^1(g),\cdot)} \ar@{-->}[d] 
&&&
\mathit{hom}_{A^{\mathit{mod}}}(M,g^*M) 
\\
\mathit{hom}_{A^{\mathit{mod}}}(M,M) \ar[rrr]^-{g^*}
&&&
\mathit{hom}_{A^{\mathit{mod}}}(g^*M,g^*M).
\ar[u]_-{\mu^2_{A^{\mathit{mod}}}(\cdot,\rho^1(g))}
}
\end{equation}
We can get slightly stronger versions of these statements by working uniformly over $G$, which means in terms of families.

\begin{lemma} \label{th:iso-family}
A weak $G$-action on $M$ yields an isomorphism $\C[G] \otimes M \rightarrow N$ in the category $H^0(A^{\mathit{mod}}/G)$.
\end{lemma}

\begin{proof}
Take \eqref{eq:2-cocycle} and pull the equality back via $(\mathit{id},i): G \rightarrow G^2$, where $i$ is the inverse map $i(g) = g^{-1}$. The result is
\begin{equation}
\mu^1_{A^{\mathit{mod}}/G}((\mathit{id},i)^*\rho^2) \\
+ \mu^2_{A^{\mathit{mod}}/G}(\gamma^* i^* \rho^1, \rho^1) \\
- \mathit{id} = 0 \in \mathit{hom}_{A^{\mathit{mod}}/G}(\C[G] \otimes M,\C[G] \otimes M).
\end{equation}
It follows that $\rho^1$ has a left inverse in $H^0(A^{\mathit{mod}}/G)$. Since $i^*$ and $\gamma^*$ are isomorphisms, it has a right inverse as well.
\end{proof}

\begin{lemma} \label{th:rational-endo}
If $M$ is a perfect module and carries a weak $G$-action, the induced $G$-action on the graded space $H^*(\mathit{hom}_{A^{\mathit{mod}}}(M,M))$ is rational.
\end{lemma}

\begin{proof}
We start with a version of \eqref{eq:endo-representation} with variable $g$, namely
\begin{equation} \label{eq:endo-representation-2}
\xymatrix{
\mathit{hom}_{A^{\mathit{mod}}/G}(\C[G] \otimes M, \C[G] \otimes M) \ar[rrr]^-{\mu^2_{A^{\mathit{mod}}/G}(\rho^1,\cdot)} \ar@{-->}[d] &&&
\mathit{hom}_{A^{\mathit{mod}}/G}(\C[G] \otimes M, N) 
\\
\mathit{hom}_{A^{\mathit{mod}}/G}(\C[G] \otimes M, \C[G] \otimes M) \ar[rrr]^-{\gamma^*}
&&&
\mathit{hom}_{A^{\mathit{mod}}/G}(N,N). \ar[u]_-{\mu^2_{A^{\mathit{mod}}/G}(\cdot,\rho^1)}
}
\end{equation}
Lemma \ref{th:iso-family} shows that the composition with $\rho^1$ on on either side is a chain homotopy equivalence, and of course $\gamma^*$ is an isomorphism. If $M$ is perfect, the cohomology groups in the left column of \eqref{eq:endo-representation-2} are isomorphic to $H(\mathit{hom}_{A^{\mathit{mod}}}(M,M)) \otimes \C[G]$. By restricting to constants, we get a $\C$-linear map
\begin{equation}
H(\mathit{hom}_{A^{\mathit{mod}}}(M,M)) \longrightarrow H(\mathit{hom}_{A^{\mathit{mod}}}(M,M)) \otimes \C[G],
\end{equation}
which is the $\C[G]$-comodule structure corresponding to the desired rational representation.
\end{proof}

Lemma \ref{th:iso-family} again implies that if a module admits a weak action, then its Killing class must vanish. More interestingly, one can obtain a partial converse. Take an $A$-module $M$ such that $\mathit{Ki}_M = 0$. Choose a bounding cochain $\alpha$ for the underlying cocycle $\mathit{ki}_M$, and associate to it a connection $\nabla_N$ as in \eqref{eq:symmetric-connection}. Equip $N^2 = m_{2,1}^*N$ with the pullback connection $m_{2,1}^*\nabla_N$, and all the constant families of modules with their trivial connections.

\begin{lemma}
Suppose that 
\begin{equation}
[\rho^1] \in H^0(\mathit{hom}_{A^{\mathit{mod}}/G}(\C[G] \otimes M,N)) = H^0(\mathit{hom}_{A^{\mathit{mod}}/G}(p_{1,1}^*N^0,N^1))
\end{equation}
is covariantly constant for the induced connection. Then, the class
\begin{equation} \label{eq:m2-obstruction-class}
[m_{2,1}^*\rho^1 - \mu^2_{A^{\mathit{mod}}/G^2}(\gamma_1^* p_{1,1}^* \rho^1, p_{2,2}^* \rho^1)] \in H^0(\mathit{hom}_{A^{\mathit{mod}}/G^2}(p_{2,1}^*N^0,N^2))
\end{equation}
has the property that its restriction to each slice $G \times \{g_1\} \subset G^2$ is covariantly constant.
\end{lemma}

\begin{proof}
By definition of the connections involved, $m_{2,1}^*\rho^1$ is covariantly constant on the whole of $G^2$. Consider the other summand and restrict it to $G \times \{g_1\}$. The restriction can be written as the composition of the following two terms:
\begin{equation} \label{eq:parts-of-covariantly-constant}
\begin{aligned}
& [g_1^*\rho^1] \in H^0(\mathit{hom}_{A^{\mathit{mod}}/G}(\C[G] \otimes g_1^*M, g_1^*N)), \\
& [\mathit{id}_{\C[G]} \otimes \rho^1(g_1)] \in H^0(\mathit{hom}_{A^{\mathit{mod}}/G}(\C[G] \otimes M, \C[G] \otimes g_1^*M)).
\end{aligned}
\end{equation}
The second of these terms is clearly covariantly constant (it is constant in the remaining variable $g = g_2$, and all the connections involved are trivial). The first is covariantly constant if we equip $N^2 | G \times \{g_1\} \iso g_1^*N$ with the connection $g_1^*\nabla_N$. That's not exactly how we described the process -- we used the geometric pullback $r_{g_1}^*\nabla_N = r_{g_1^{-1},*}\nabla_N$ instead -- but the result coincides with $g_1^*\nabla_N$, as discussed in Remark \ref{th:2-pullbacks}.
We now know that both expressions \eqref{eq:parts-of-covariantly-constant} are covariantly constant, hence so is their product by \eqref{eq:product-of-connections}.
\end{proof}

\begin{lemma} \label{th:cov-const-2}
Suppose that $A$ is weakly proper, and that $M$ is a perfect $A$-module. Take a $[\rho^1]$ as in the previous Lemma, with the additional assumption that its value at $g = e$ should be the identity. Then $[\rho^1]$ is automatically a weak action.
\end{lemma}

\begin{proof}
Weak properness implies that $H(N^r)$ is a finitely generated $\C[G^r]$-module in each degree. The constant family $p_{r,1}^*N^0$ is perfect, hence (by Lemma \ref{th:finite-generation}) $H^*(\mathit{hom}_{A^{\mathit{mod}}/G^r}(p_{r,1}^*N^0,N^r))$ is again finitely generated in each degree. Moreover, it admits a connection, hence is projective. We want to show that \eqref{eq:m2-obstruction-class} vanishes. It suffices to show that its restriction to each slice $G \times \{g_1\}$ is zero. We know that each such restriction is covariantly constant, and that its value at $(e,g_1)$ is
\begin{equation}
[\rho^1(g_1) - \mu^2_{A^{\mathit{mod}}}(g_1^*\rho^1(e),\rho^1(g_1))] = 0. 
\end{equation}
These two facts together imply the necessary vanishing result.
\end{proof}

Here is one concrete consequence of the discussion so far:

\begin{lemma} \label{th:exceptional}
Suppose that $A$ is weakly proper, and that $M$ is a perfect $A$-module which is rigid and simple, meaning that it satisfies the analogue of \eqref{eq:exceptional} in $A^{\mathit{mod}}$. Then $M$ can be equipped with a weak $G$-action.
\end{lemma}

\begin{proof}
$\mathit{Ki}_M$ necessarily vanishes, which means that $H^0(\mathit{hom}_{A^{\mathit{mod}}/G}(\C[G] \otimes M,N))$ carries a connection, hence is locally free of rank $1$. It is therefore necessarily free as a $\C[G]$-module. When choosing a coboundary $\alpha$ for $\mathit{ki}_M$, we have the freedom of adding any scalar multiple of $\mathit{id}_M$ to it, and that changes the induced connection on $H^0(\mathit{hom}_{A^{\mathit{mod}}/G}(\C[G] \otimes M,N))$ by the corresponding scalar factor. By adjusting that factor, one can achieve that the monodromy is trivial (here we use the surjectivity of the exponential map, hence the fact that our ground field is $\C$), and then there is a covariantly constant section as required by Lemma \ref{th:cov-const-2}.
\end{proof}

\section{Homotopy actions}

\subsection{Definition}
Starting from the notion of weak action, one is naturally led to introduce higher homotopies following the familiar simplicial pattern. We will now explain the resulting notion. As before, $A$ is an $A_\infty$-algebra carrying a rational action of $G = \C^*$.

Let $M$ be an $A$-module. A {\em homotopy action} of $G$ on $M$ is given by maps
\begin{equation} \label{eq:h-action-1}
\rho^{r,d+1}(g_r,\dots,g_1): M \otimes A^{\otimes d} \longrightarrow M[1-d-r]
\end{equation}
for $r \geq 1$, $g_r,\dots,g_1 \in G$, and $d \geq 0$, which satisfy conditions which we will now gradually introduce. Most importantly, the {\em cocycle condition} says that
\begin{equation} \label{eq:homotopy-cocycle}
\begin{aligned}
&  \sum_i (-1)^{|a_{i+1}| + \cdots + |a_d| + |m| + d-i} \mu_M^{i+1}(\rho^{r,d+1-i}(g_r,\dots,g_1,m,a_d,\dots,a_{i+1}),
\\[-1em] & \qquad \qquad \qquad \qquad \qquad \qquad \qquad \qquad \qquad \qquad \qquad
g_r \cdots g_1(a_i),\dots,g_r \cdots g_1(a_1)) \\
 + & \sum_i (-1)^{|a_{i+1}| + \cdots + |a_d| + |m| + d-i} \rho^{r,i+1}(g_r,\dots,g_1,\mu_M^{d-i+1}(m,a_d,\dots,a_{i+1}),a_i,\dots,a_1) \\
 + & \sum_{i,j} (-1)^{|a_{i+1}| + \cdots + |a_d| + |m| + d-i} \rho^{r,d+2-j}(g_r,\dots,g_1,m,a_d,\dots,a_{i+j+1}, 
\\[-1em] & \qquad \qquad \qquad \qquad \qquad \qquad \qquad \qquad \qquad \qquad \qquad 
\mu_A^j(a_{i+j},\dots,a_{i+1}),a_i,\dots,a_1) \\ 
+ & \sum_{q,i} (-1)^{|a_{i+1}| + \cdots + |a_d| + |m| + d-i} \rho^{r-q,i+1}(g_r,\dots,g_{q+1}, 
\\[-1em] & \qquad \qquad \qquad \qquad \rho^{q,d+1-i}(g_q,\dots,g_1,m,a_d,\dots,a_{i+1}),g_q \cdots g_1(a_i),\dots,g_q \cdots g_1(a_1))
 \\
+ & \sum_q (-1)^q \rho^{r-1,d+1}(g_r,\dots,g_{q+1}g_q,\dots,g_i,m,a_d,\dots,a_1) = 0. \\
\end{aligned}
\end{equation}
The $r = 1$ and $r = 2$ cases simply reproduce \eqref{eq:cocycle1} and \eqref{eq:2-cocycle}. For each $(g_r,\dots,g_1)$, one can consider the collection $\{\rho^{r,d+1}(g_r,\dots,g_1,\dots)\}_{d \geq 0}$ as an element
\begin{equation}
\rho^r(g_r,\dots,g_1) \in \mathit{hom}_{A^{\mathit{mod}}}^{1-r}(M, g_1^* \dots g_r^* M),
\end{equation}
and then \eqref{eq:homotopy-cocycle} can be rewritten in the same way as in \eqref{eq:1-cocycle-condition}:
\begin{equation} \label{eq:entire-cocycle-short}
\begin{aligned}
& \mu^1_{A^{\mathit{mod}}}(\rho^r(g_r,\dots,g_1)) \\ + & \sum_q \mu^2_{A^{\mathit{mod}}}(g_1^* \dots g_q^* \rho^{r-q}(g_r,\dots,g_{q+1}), \rho^q(g_q,\dots,g_1) ) \\
+ & \sum_q (-1)^q \rho^{r-1}(g_r,\dots,g_{q+1} g_q,\dots,g_1) = 0.
\end{aligned}
\end{equation}
The {\em unitality} condition is the same as for weak actions, saying that $[\rho^1(e)]$ is the identity in $H^0(A^{\mathit{mod}})$. The {\em rationality} condition says that if we fix $d$, the $\rho^{r,d+1}(g_r,\dots,g_1,\dots)$ for different $(g_r,\dots,g_1)$ should be specializations of a single map
\begin{equation} \label{eq:rho-maps}
\rho^{r,d+1}: M \otimes A^{\otimes d} \longrightarrow \C[G^r] \otimes M[1-d-r].
\end{equation}
Using \eqref{eq:family-hom} this can be thought of as a morphism in the $A_\infty$-category $A^{\mathit{mod}}/G^r$, 
\begin{equation}
\rho^r \in \mathit{hom}_{A^{\mathit{mod}}/G^r}^{1-r}(p_{r,1}^*N^0,N^r),
\end{equation}
and then \eqref{eq:entire-cocycle-short} can be written in analogy to \eqref{eq:2-cocycle-short} as
\begin{equation} \label{eq:entire-cocycle-2}
\mu^1_{A^{\mathit{mod}}/G^r}(\rho^r) 
+ \sum_q \mu^2_{A^{\mathit{mod}}/G^r}(\gamma_1^* \dots \gamma_q^* p_{q,1}^*\rho^{r-q},  p_{r,q+1}^{*} \rho^q)
+ \sum_q (-1)^q m_{q+1,q}^*\rho^{r-1} = 0.
\end{equation}
We summarize the discussion:

\begin{definition}
A {\em homotopy action} of $G$ on $M$ is a collection of maps \eqref{eq:rho-maps} satisfying the unitality and cocycle conditions.
\end{definition}

\subsection{Turning homotopy actions into naive ones}
Given a homotopy $G$-action on $M$, we can construct another $A_\infty$-module $M^{\mathit{naive}}$ whose underlying graded vector space is
\begin{equation} \label{eq:strictification}
M^{\mathit{naive}} = \prod_{r \geq 0} \C[G]^{\otimes r+1} \otimes M[-r].
\end{equation}
If we think of elements as collections of maps $\beta^{r+1}: G^{r+1} \rightarrow M$, then the differential is defined to be
\begin{equation}
\begin{aligned}
\mu^1_{M^{\mathit{naive}}}(\beta)^{r+1}(g_{r+1},\dots,g_1) & =
 \mu^1_M(\beta^{r+1}(g_{r+1},\dots,g_1)) \\ & +
\sum_q \rho^{r+1-q,1}(g_{r+1},\dots,g_{q+1},\beta^q(g_q,\dots,g_1)) \\ & +
\sum_q (-1)^{q+|\beta|} \beta^r(g_{r+1},\dots,g_{q+1}g_q,\dots,g_1),
\end{aligned}
\end{equation}
and the other $A_\infty$-module structures as
\begin{equation}
\begin{aligned}
& \mu^{d+1}_{M^{\mathit{naive}}}(\beta,a_d,\dots,a_1)^{r+1}(g_{r+1},\dots,g_1) =
\\ & \qquad
= \mu_M^{d+1}(\beta^{r+1}(g_{r+1},\dots,g_1),g_{r+1}\dots g_1(a_d),\dots,g_{r+1}\dots g_1(a_1))
\\ & \qquad
+ \sum_q \rho^{r+1-q,d}(g_{r+1},\dots,g_{q+1},\beta^q(g_q,\dots,g_1),g_q \cdots g_1(a_d),\dots, g_q\cdots g_1(a_1)).
\end{aligned}
\end{equation}
Take the (complete decreasing) filtration of $M^{\mathit{naive}}$ by $r$. The associated spectral sequence has $E_1^{r*} = \C[G]^{\otimes r+1} \otimes H(M)$. The next differential takes a given representative $[\beta^r]$ to
\begin{equation}
\begin{aligned}
(g_{r+1},\dots,g_1) \longmapsto&  (-1)^{|\beta(g_r,\dots,g_1)|+r-1} \rho^{1,1}(g_{r+1},\beta^r(g_r,\dots,g_1)) \\ & \quad + \sum_q (-1)^q \beta^r(g_r,\dots,g_{q+1}g_q,\dots,g_1),
\end{aligned}
\end{equation}
hence is precisely the differential \eqref{eq:delta-bar} for the $G$-action on $H(M)$ induced by $\rho^{1,1}$. It follows that the only nonvanishing term on the next page is $E_2^{0*} \iso H(M)$. In fact, by applying Lemma \ref{th:bar} and a spectral sequence comparison theorem, one sees that the module homomorphism $\phi \in \mathit{hom}_{A^{\mathit{mod}}}(M,M^{\mathit{naive}})$ given by
\begin{equation}
\phi^{d+1}(m,a_d,\dots,a_1)(g_{r+1},\dots,g_1) = \rho^{r+1,d+1}(g_{r+1},\dots,g_1,m,a_d,\dots,a_1)
\end{equation}
is a quasi-isomorphism. One also has a quasi-isomorphism in inverse direction, which is linear, and takes $\beta$ to $m = \beta^1(e)$.

$M^{\mathit{naive}}$ carries a $G$-action on the underlying graded vector space, mapping $\beta(g_{r+1},\dots,g_1)$ to $\beta(g_{r+1},\dots,g_1g)$, which is compatible with all the $A_\infty$-operations. Because of the direct product in \eqref{eq:strictification}, this is not in general a rational representation. However, if $M$ is bounded below as a graded vector space, then each graded piece of $M^{\mathit{naive}}$ is a finite product, hence the problem does not arise:

\begin{lemma} \label{th:strictification}
Suppose that $M$ is bounded below. Then, if it carries a homotopy $G$-action, there is a quasi-isomorphic module which carries a naive $G$-action. \qed
\end{lemma}

%

\subsection{Obstruction theory}
We will now consider the issue of extending a weak action to a homotopy action. Fix some $M$ and some $s \geq 3$. Suppose that we are given maps $\rho^r$ for all $r < s$, satisfying the unitality condition as well as those equations \eqref{eq:entire-cocycle-2} in which no higher order maps appear. Temporarily set $\rho^s = 0$. Then, the failure of the order $s$ equation \eqref{eq:entire-cocycle-2} to hold is measured by a collection of maps
\begin{equation} \label{eq:error}
\begin{aligned}
\epsilon^s(g_s,\dots,g_1) = & \sum_q \mu^2_{A^{\mathit{mod}}}(g_1^* \dots g_q^* \rho^{r-q}(g_r,\dots,g_{q+1}), \rho^q(g_q,\dots,g_1) ) ) \\
+ & \sum_q (-1)^q \rho^{r-1}(g_r,\dots,g_{q+1} g_q,\dots,g_1) 
\in \mathit{hom}_{A^{\mathit{mod}}}^{2-s}(M, g_1^* \dots g_s^* M).
\end{aligned}
\end{equation}
A straightforward computation starting with \eqref{eq:entire-cocycle-2} shows that these are $\mu^1_{A^{mod}}$-closed, hence actual module homomorphisms. Moreover,
\begin{equation} \label{eq:obstruction-cocycle}
\begin{aligned}
& -\mu^2_{A^{\mathit{mod}}}(g_1^*\dots g_s^* \rho^1(g_{s+1}),\epsilon^s(g_s,\dots,g_1))
+ \mu^2_{A^{\mathit{mod}}}(g_1^*\epsilon^s(g_{s+1},\dots,g_2),\rho^1(g_1)) \\
& \qquad \qquad + \sum_q (-1)^q \epsilon^s(g_{s+1},\dots,g_{q+1}g_q,\dots,g_1) = \mu^1_{A^{\mathit{mod}}}(\text{\it something}).
\end{aligned}
\end{equation}
Equivalently, one can consider the class
\begin{equation} \label{eq:error2}
E^s(g_s,\dots,g_1) = (g_1^{-1} \cdots g_s^{-1})^* ([\epsilon^s(g_s,\dots,g_1))] \cdot [\rho^1(g_s \dots g_1)]^{-1} \in 
H^{2-s}(\mathit{hom}_{A^{\mathit{mod}}}(M,M)).
\end{equation}
Taking into account the sign changes \eqref{eq:cohomology} when passing to the cohomological category, composition with $[\rho^1(g_{s+1}\dots g_1)]^{-1}$ and pullback turns \eqref{eq:obstruction-cocycle} into
\begin{equation} \label{eq:error3}
\begin{aligned}
& (-1)^{s+1} g_{s+1}^* ( [\rho^1(g_{s+1})] E^s(g_s,\dots,g_1) [\rho^1(g_{s+1})]^{-1} ) + E^s(g_{s+1},\dots,g_2) \\ & \qquad \qquad + \sum_q (-1)^q E^s(g_{s+1},\dots,g_{q+1}g_q,\dots,g_1) = 0.
\end{aligned}
\end{equation}
The first term (modulo sign) is the action of $g_{s+1}$ on $E^s(g_s,\dots,g_1) \in H^{2-s}(\mathit{hom}_{A^{\mathit{mod}}}(M,M))$ as defined in \eqref{eq:endo-representation}. Considering that cohomology group as a $G$-module, we may write \eqref{eq:error3} as a group cocycle equation
\begin{equation}
\delta_{C^*(G,H^{2-s}(\mathit{hom}_{A^{\mathit{mod}}}(M,M)))}(E^s) = 0.
\end{equation}
To make this argument precise, we need to work with varying elements of $G$, which means that we consider the error term and its associated cohomology class as elements
\begin{equation}
\begin{aligned}
& \epsilon^s \in \mathit{hom}_{A^{\mathit{mod}}/G^s}^{2-s}(p_{s,1}^* N^0, N^s), \\
& E^s \in H^{2-s}(\mathit{hom}_{A^\mathit{mod}}(M, M \otimes \C[G^s])).
\end{aligned}
\end{equation}
Assume at this point that $M$ is perfect. Using \eqref{eq:compact} one can rewrite the second term above as
\begin{equation}
E^s \in H^{2-s}(\mathit{hom}_{A^{\mathit{mod}}}(M,M)) \otimes \C[G]^{\otimes s},
\end{equation}
which by \eqref{eq:error3} is a cocycle of degree $s$ in the complex $C^*(G,H^{2-s}(\mathit{hom}_{A^{\mathit{mod}}}(M,M)))$. A computation similar to the previous ones (and whose details we therefore omit) shows that, by adding a cocycle in $\mathit{hom}_{A^{\mathit{mod}}/G^{s-1}}(p_{s-1,1}^*N^0,N^{s-1})$ to $\rho^{s-1}$, one can change $E^s$ by an arbitrary coboundary, without affecting the validity of the equations \eqref{eq:entire-cocycle-2} for $r<s$.

We know from Lemma \ref{th:cohomology} that the group cochain complex is acyclic in the relevant degree. This means that by adjusting $\rho^{s-1}$, one can get $E^s$ to become zero, which means that $\epsilon^s$ is a coboundary. After that, one can choose a $\rho^s$ so that the order $s$ equation \eqref{eq:entire-cocycle-2} is satisfied. Induction then shows the following:

\begin{lemma} \label{th:extend}
Every weak $G$-action on a perfect $M$ can be extended to a homotopy $G$-action. \qed
\end{lemma}

We can now state and prove a slightly more general version of Theorem \ref{th:main4}.

\begin{corollary} \label{th:proper}
Let $A$ be an $A_\infty$-algebra with a rational $G$-action, such that $H(A)$ is finite-dimensional in each degree (weak properness) and bounded below. Let $M$ be a perfect $A_\infty$-module over $A$ satisfying the analogue of \eqref{eq:exceptional} in $A^{\mathit{mod}}$. Then there is a quasi-isomorphic $A_\infty$-module which carries a naive $G$-action.
\end{corollary}

\begin{proof}
Using standard Perturbation Lemma arguments, we can replace $A$ by a quasi-isomorphic $A_\infty$-algebra which is minimal (has vanishing differential). This works equivariantly with respect to $G$, so we will assume from now on that $A$ itself is minimal. Similarly, we can replace $M$ by a quasi-isomorphic module which is minimal, and we will assume that this has been done as well. Since $M$ is perfect, it then follows that it must be finite-dimensional in each degree and bounded below. One uses Lemma \ref{th:exceptional} to equip $M$ with a weak $G$-action; Lemma \ref{th:extend} to extend that to a homotopy $G$-action; and finally Lemma \ref{th:strictification} to convert that into a naive $G$-action on a quasi-isomorphic module.
\end{proof}

%
\providecommand{\bysame}{\leavevmode\hbox to3em{\hrulefill}\thinspace}
\providecommand{\MR}{\relax\ifhmode\unskip\space\fi MR }
\providecommand{\MRhref}[2]{%
  \href{http://www.ams.org/mathscinet-getitem?mr=#1}{#2}
}
\providecommand{\href}[2]{#2}

\end{document}